\documentclass{article}
\pdfoutput=1

\usepackage[utf8]{inputenc}
\usepackage{amsmath, amsthm, amssymb, mathtools}
\usepackage{amsfonts}
\usepackage{bm}
\usepackage{cancel}
\usepackage{color}
\usepackage{subfig}
\usepackage{stmaryrd}
\usepackage{xparse} 
\usepackage{nomencl} 
\usepackage{cleveref}
\usepackage{empheq} 
\usepackage{enumerate}
\usepackage{cases}
\usepackage{lineno}

\usepackage{textcomp} 
\usepackage{stmaryrd}
\usepackage{graphicx}
\usepackage{float}
\usepackage[export]{adjustbox}

\newcommand{\tred}{\textcolor{black}} 
\newcommand{\tblue}{\textcolor{black}} 
\newcommand{\tgre}{\textcolor{black}} 

\numberwithin{equation}{section}

\newcommand{\hmbu}{ \hat{\bm u} }
\newcommand{\tV}{ \tilde{V} }
\newcommand{\mbr}{\bm r}
\newcommand{\mbR}{\bm R}

\newcommand{\mbw}{\bm w}
\newcommand{\mbz}{\bm z}
\newcommand{\mbph}{\bm \phi}
\newcommand{\mbx}{\bm x}
\newcommand{\mbxi}{\bm \xi}
\newcommand{\mby}{\bm y}
\newcommand{\mbu}{\bm u }

\newcommand{\mbv}{\bm v}
\newcommand{\mbg}{\bm g }
\newcommand{\mbn}{\bm n}

\newcommand{\mbF}{\bm F }
\newcommand{\mbf}{\bm f }

\newcommand{\tW}{\tilde{W} }

\newcommand{\la}{\lambda}

\newcommand{\mbphi}{\bm \phi }

\DeclarePairedDelimiter{\norm}{\lVert}{\rVert}
\DeclarePairedDelimiter{\jumpop}{\llbracket}{\rrbracket}
\DeclareSymbolFont{matha}{OML}{txmi}{m}{it}
\DeclareMathSymbol{\varv}{\mathbf}{matha}{118}

\newtheorem{prop}{Proposition}[section]

\newtheorem{Theorem}{Theorem}[section]
\newtheorem{Definition}{Definition}[section]
\newtheorem{Lemma}{Lemma}[section]
\newtheorem{Corollary}{Corollary}[section]
\newtheorem{Remark}{Remark}[section]

\newcommand\restr[2]{{
  \left.\kern-\nulldelimiterspace 
  #1 
  \vphantom{\big|} 
  \right|_{#2} 
  }}

\NewDocumentCommand{\dgal}{sO{}m}{%
  \IfBooleanTF{#1}
    {\dgalext{#3}}
    {\dgalx[#2]{#3}}%
}

\NewDocumentCommand{\dgalext}{m}{%
  \sbox0{%
    \mathsurround=0pt 
    $\left\{\vphantom{#1}\right.\kern-\nulldelimiterspace$%
  }%
  \sbox2{\{}%
  \ifdim\ht0=\ht2
    \{\kern-.625\wd2 \{#1\}\kern-.625\wd2 \}%
  \else
    \left\{\kern-.7\wd0\left\{#1\right\}\kern-.7\wd0\right\}%
  \fi
}

\NewDocumentCommand{\dgalx}{om}{%
  \sbox0{\mathsurround=0pt$#1\{$}%
  \sbox2{\{}%
  \ifdim\ht0=\ht2
    \{\kern-.625\wd2 \{#2\}\kern-.625\wd2 \}%
  \else
    \mathopen{#1\{\kern-.7\wd0 #1\{}
    #2
    \mathclose{#1\}\kern-.7\wd0 #1\}}
  \fi
}

\newcommand{\R}{{\mathbb{R}}}

\crefformat{section}{\S#2#1#3} 
\crefformat{subsection}{\S#2#1#3}
\crefformat{subsubsection}{\S#2#1#3}

\title{\bf Approximations of energy minimization in cell-induced phase transitions of fibrous biomaterials:
$\Gamma$-convergence analysis
}

\author{G.  Grekas\thanks{Aerospace Engineering and Mechanics, University of Minnesota, Minneapolis, USA.}
\and K. Koumatos\thanks{Department of Mathematics, University of Sussex.}
\and C. Makridakis\footnotemark[4] \footnotemark[2] \thanks{Department of Mathematics and Applied Mathematics, University of Crete.} 
\and P.\ Rosakis\footnotemark[3] \thanks{Institute of Applied \& Computational Mathematics, Foundation for Research \& Technology-Hellas}. 
}

\date{}
\begin{document}
\maketitle


\begin{abstract}

We consider a model of energy minimization arising in the study of  the mechanical behavior caused by cell contraction within a fibrous \tgre{biological medium}.  
The macroscopic model is based on the theory of \tblue{non rank-one convex} nonlinear elasticity for phase transitions. We study  appropriate numerical approximations based on the discontinuous Galerkin treatment of higher gradients and  used  \tgre{succesfully  in numerical simulations of  experiments}.  We show that the discrete minimizers  converge in the limit to  minimizers of the continuous problem. 
This is achieved by employing the theory of $\Gamma$-convergence  of the approximate energy  functionals to
 the continuous model when the discretization parameter tends to zero. 
The analysis is involved  due to the structure of numerical approximations which are defined in spaces with lower regularity than the space where  the minimizers of the continuous variational problem are sought. This  fact  
leads to the development of a new approach to $\Gamma$-convergence, appropriate for discontinuous finite element discretizations, which can be applied to quite general energy minimization problems. 
Furthermore, the adoption of exponential terms penalising the interpenetration of matter requires  a new framework based on Orlicz spaces for discontinuous Galerkin methods which is developed in this paper as well.
\end{abstract}

\section{Introduction}
Increasingly sophisticated mathematical techniques are needed in 
order to describe biological phenomena. \tblue{Biological cells severely deform the fibrous collagen extracellular material (ECM) around them in remarkable ways, causing the formation
of complex microstructures of highly localized deformation \cite{harris1981fibroblast,stopak1982connective}. In \cite{Grekas_2021}, this phenomenon, known as mechanical remodelling of the ECM, is modelled at a macroscopic level, as a phase transition in a nonlinear elastic material with a multi-well energy. The nonconvexity is due to buckling instability of ECM fibers at the microsopic level. Subsequently,  computational predictions  based on the mathematical model, combined with targeted experiments, lead for the first time to 
understanding the mechanisms of the observed ECM remodelling. The latter facilitates intercellular communication through the formation of tethers, regions where a densification phase transition takes place,  joining distant contracting cells. The  associated variational problem involves a non rank-one convex strain-energy
function, regularized by a higher gradient term.}  

\tblue{Here,} our objective is to mathematically justify the above procedure, 
by showing that appropriate numerical approximations, based on the discontinuous Galerkin treatment of higher gradients, and used very successfully in   computational  experiments, indeed converge in the limit to  minimizers of the continuous problem. 
This is done by employing the theory of $\Gamma$-convergence  of the approximate energy   functionals to
 the continuous model when the discretization parameter tends to zero. 
This is a rather involved task due to the structure of numerical approximations, which are defined in spaces with lower regularity than the space where the minimizers of the continuous variational problem are sought.

Our work has a number of methodological advances which go beyond the scope of the particular application. In fact, within the field of nonlinear PDEs with possibly singular solutions, calculus of variations and energy minimization has received a lot of attention from the analysis point of view. Although quite interesting and challenging,  the numerical analysis of these problems is much less developed,   {a variety of 
approaches 
are discussed in,  e.g., 
\cite{bartels2017bilayer, 
Bartels_relaxation2004,  
Carstensen_Plec_1997, Luskin_1998, Pedregal_1996}
and their  references.} Key issues, related to the design of computational algorithms, their analysis, selection criteria among possible solutions and minimizers remain largely  unexplored. Our work  contributes to scheme design and analysis of such problems. The problem considered is a typical nonlinear energy minimization problem which admits solutions exhibiting  phase transitions. 
 We show that  a good scheme design strategy is  to regularise at the discrete level the continuum energies by higher order gradients. The regularisation at the discrete level considered is of  similar nature to the artificial diffusion in conservation laws and it appears that \tgre{it enjoys} remarkable properties in terms of computational  robustness and  analytical consistency.
  \tblue{Compared to other approaches, such as relaxation, \cite{Bartels_relaxation2004}, the regularisation by higher gradients  has certain distinct desirable characteristics. 
 Computing with the relaxed energy, assuming that we can obtain it,   smears over the multiphase mixture (microstructure) or oscillations, which in certain applications (including ours) are very physical as they are observed   in   experiments. Furthermore, when microstructures appear, quite often a mathematical  quantity of interest is an underlying parametrized Young measure. The regularisation parameter $\varepsilon$  sets an upper bound to the frequency of the oscillations and thus allows micro-phenomena to be present in a computationally accessible way;  thus such oscillations (not necessarily extremely fine) 
 may provide approximations to the underlying Young measure.} 
 The discretization of higher gradients is done using the discontinuous Galerkin framework, thus retaining the regularity of only $C^0$ conforming elements. This natural approach on the other hand poses new challenges in the analysis of schemes.  New ideas are needed in the $\Gamma$-convergence 
 analysis due to the lack of conformity at the higher-gradient spaces. The adoption of exponential terms penalising the interpenetration of matter requires  a new framework based on Orlicz spaces for discontinuous Galerkin methods, developed in this paper as well. 

\tblue{Soon after a preprint version of this work was published \cite{grekas2019approximations}, a related preprint \cite{bonito2019dg} appeared, which also treats $\Gamma$-convergence of
the discetized energy from a nonconvex mechanics problem, using totally discontinuous finite elements (these two works were developed independently of each other). Apart from this similarity, there are substantial differences in both the model and the discretization
that set these papers apart.}

\noindent
\emph{The model.\/} We consider the problem of minimizing the total potential energy
\begin{equation}
\begin{aligned}
\Psi[\mbu] = \int_{\Omega} \left[ W(\nabla \mbu(\mbx))  + \Phi(\nabla \mbu(\mbx)) +  \frac{\varepsilon^2}{2} |
\nabla \nabla \mbu(\mbx)|^2 \right]  d\mbx,
\end{aligned}
 \label{equ:total_potential}
\end{equation}
\tblue{where the displacement} $\mbu \in H^2(\Omega)^2$ \tblue{and} satisfies some appropriate boundary conditions,
$W+\Phi$ is the strain energy function, 
$\Phi$ is a function that penalizes the interpenetration
of matter and is allowed to grow faster than $W$ as the volume ratio approaches zero, and $\varepsilon >0$ is a fixed real parameter (higher gradient coefficient).
%
%
%
%
The energy involves  a non rank-one convex strain-energy
function, regularized by a higher order term.  The penalty term $\Phi $ is important since, 
although it permits the appearance of  phase transitions, 
it prevents interpenetration \tblue{of matter} from taking place. 
The strain energy function models the bulk response
of the collagen ECM, while the higher order term represents a length scale for the thickness of phase transition layers and the emerging two-phase microstructures.
Specifically, the strain energy function models the mechanical response of the 
 \tblue{ECM} which
is  a   collagen material in the form of a random network of fibers at the microscopic level.
Biological cells, such as fibroblasts, are embedded in the ECM. They are attached onto the ECM fibers through proteins known as focal adhesions.
Through these molecules, cells can detect mechanical alterations to their microenvironment and can 
deform the surrounding ECM. Cells typically deform the ECM by actively 
contracting. These tractions are observed to create distinct spatial patterns of \tblue{localized, severe} densification
in the ECM between cells, forming a tether connecting them, and around the periphery of the cell
in the form of hair-like microstructures, \cite{harris1981fibroblast,stopak1982connective,Notbohm2015, Grekas_2021}.
 \tblue{These microstructures and the associated strain oscillations are strongly reminiscent of fine phase mixtures in the theory of nonlinear elasticity for phase transitions \cite{Ball1987,ball1999compatibility}. This similarity was explained in \cite{Grekas_2021} where a model for the strain energy density was obtained via multiscale modelling from the stress-strain behavior of single fibers comprising the bulk  ECM material, as summarized below in section \ref{sec:computations}. The central physical observation related to instability and phase change \cite{lakes1993microbuckling} is that} individual collagen fibers can sustain tension, but buckle and collapse under compression. \tblue{At the microscopic level, this is accounted for as an effective softening behavior in uniaxial compression of fibers in our model. The macroscopic behavior is obtained by  averaging over a uniform distribution of fiber orientations. This results in a strain energy function that loses rank-one convexity, and is essentially equivalent to a multi-well potential corresponding to} a densification phase transition at the continuum scale.
The phases correspond to low- and high-density states and their simulation \tblue{via energy minimization} leads to a remarkable agreement with experimental observations of densification microstructures \tblue{(an example of an experiment
 is shown in Figure~\ref{fig:fn3a}; the corresponding simulation is shown in Figure~\ref{fig:fn3b}). These microstructures are composed  of tethers, or relatively straight  bands joining different cells, and hairs,  thinner multiple bands emanating radially from each cell and tapering off into the ECM. The ECM density within both tethers and hairs can be  3-5 times larger than outside them.}
 
Failure of rank-one convexity implies a loss of ellipticity of the Euler-Lagrange PDEs of the corresponding
energy functional.  For a wide class of similar problems it is known  \cite{Ball1987,ball1999compatibility} that there exist oscillatory minimizing sequences with
finer and finer microstructures involving increasing numbers of strain jumps. \tblue{Similar behavior is observed in our model; the numerical approximations---obtained by mesh refinement---of terms of increasing fineness in the minimizing sequences, involve more and thinner hairs (see \tgre{Figures~\ref{fig:h0}-\ref{fig:h0/8}}). They also bear strong similarity with experiments \cite{Grekas_2021}.} 

The higher gradient term in \eqref{equ:total_potential} regularises  the corresponding
total potential energy, keeping the aforementioned minimizing sequences from having arbitrarily fine structure \tblue{(see Figure~\ref{fig:varying_e} for numerical examples)}. 

To the best of our knowledge, the deformations observed in \cite{Grekas_2021} and in the present study 
are the first examples of minimizing sequences in a multi-well compressible isotropic material.
\\

\noindent
\emph{Approximations and results.\/} The approximation of  minimizers of $(\ref{equ:total_potential})$  is quite subtle.  A straightforward  approach would be to seek approximate minimizers in the space of 
\emph{conforming} finite elements, i.e. of discrete  function spaces which are  finite dimensional subspaces of
$H^2(\Omega)^2$.  Such spaces are based on elements which require $C^1$ regularity across element interfaces,
e.g. Argyris elements \cite{brenner2007mathematical}.
However, the conformity in regularity has a very high computational 
cost under the minimization process and in addition results  in much more  complicated  algorithms as far as the implementation is concerned. 
%
%
%
%
Our choice is to use the framework of the discontinuous Galerkin method. In effect this  weakens the regularity of the approximating spaces, and  counterbalances the resulting nonconformity, by amending appropriately the discrete energy functional. Motivated by the analysis in  \cite{makridakis2014atomistic}, we introduce an approximate energy which is compatible with  $C^0$ finite element spaces, and thus requires only $H^1$ regularity.  
Corresponding finite element methods, known as $C^0$-interior penalty methods,  have been introduced  previously for the approximation of the biharmonic equation in \cite{brenner2005c0,engel2002continuous}; 
see also  \cite {Baker_1977} for fully discontinuous finite element methods.  

Here we study the convergence of discrete absolute minimizers. 
Specifically, let  $(\mbu_h)$ be a sequence of absolute minimizers for the discretized energy functional $\Psi_h$,
 namely
\begin{align}
\Psi_h[\mbu_h] = \inf_{\mbw_h \in \mathbb{A}^q_h(\Omega)} \Psi_h[\mbw_h]. 
\label{eq:abs_minimizer1}
\end{align}
Equation (\ref{eq:abs_minimizer1}) indicates that, for a fixed $h$, 
$\mbu_h$ is an absolute minimizer of $\Psi_h$. Therefore, as $h \rightarrow 0$, it is natural to
ask whether  $\mbu_h \rightarrow \mbu$ in $H^1(\Omega)^2$, where 
$\mbu$ is an absolute minimizer of the continuous problem (\ref{equ:total_potential}).
Note that $\mbu_h \in H^1(\Omega)^2$ and $\mbu \in H^2(\Omega)^2$.
To answer this question we assume first that the penalty function $\Phi$ has polynomial 
growth. Then the convergence result is given in Theorem~\ref{Thm:min_convergence}, 
where we have employed the theory of $\Gamma-$convergence and discrete
compactness results.  The analysis is rather involved due to the lack of regularity of the approximate spaces. A   $\Gamma$-convergence  result  for  discrete surface functionals involving high gradients  using 
conforming finite element spaces can be found in \cite{bartels2017bilayer}. 
Assuming that the penalty term $\Phi$ has exponential growth, extra embedding results are 
needed to show that  $\Psi_h$ $\Gamma-$converges to $\Psi$.  For this purpose,
it is crucial to use  an adaptation of  Trudinger's embedding theorem for Orlicz 
spaces \cite{trudinger1967imbeddings}, to the  piecewise  polynomial spaces admitting discontinuities in the gradients,
Theorem~\ref{Orlicz_dg}.  The analysis in this case is carried out in Section 7. \tblue{It is to be noted that in our approach we are interested in the limit $h \to 0$ for any fixed $\varepsilon.$ The tools to address the very interesting case $\varepsilon \to 0$  for  rather general  non rank-one convex  functionals  as $W$  considered herein, are currently not available. }
In addition, we remark that the case where $\Phi(\nabla \mbu) \rightarrow \infty$ as $\det(\bm{1} + \nabla \mbu)\to 0$
is currently beyond our reach.

 The present work addresses key technical issues related to the analysis of approximations of energy minimization problems involving higher gradients. A main technical obstacle in proving $\Gamma$-convergence  is the  fact that  the approximating discrete energy functionals are defined in spaces of lower regularity compared to the limiting functional as a result of the discontinuous Galerkin formulation. Notice here that 
from a computational perspective the use of $C^0$ elements permits direct comparisons with the approximations obtained even in the limit case $\varepsilon=0.$ In our analysis we use certain recovery operators for the higher gradients, well known in the analysis of discontinuous Galerkin methods. As such,  previous results from 
\cite{buffa2009compact, di2010discrete} are useful. Notice that compared to the method in  \cite{buffa2009compact} where recovery operators were used in the definition of the discretization method as well, our method leads to the natural discontinuous Galerkin method, in the sense that the part corresponding to the higher gradients in the energy introduced herein has as first variation
 the bi-linear form used in \cite{brenner2005c0}.  As mentioned, to treat exponential penalty terms in the energy functional we need to develop  an appropriate  discontinuous Galerkin framework in Orlicz spaces (Section 7).
 To this end, results of \cite{karakashian2003posteriori}  for averaging operators  have proven useful.

This paper is organized as follows. In \cref{sec:continuousProblem} we discuss some 
properties of the continuum problem, a lower bound  is proved and the minimization problem 
is stated. 
In \cref{sec:discretization} the necessary
notation, some standard finite elements results, the discrete total potential energy $\Psi_h$ and
lifting operators used in the next sections are introduced.  Equi-coercivity, the $\liminf$ and the
$\limsup$ inequalities are proved for the  discrete energy functional in \cref{sec:G_convergence},
which imply $\Psi_h \xrightarrow{\enskip \Gamma \enskip}  \Psi$. From the 
$\Gamma$-convergence result  and a discrete compactness property we deduce
the convergence of  the discrete absolute minimizers, 
section~\ref{sec:compactness_convergence}.
In section \ref{sec:Orlicz_embedding} the same convergence result is established 
when the  penalty function has exponential growth. In this section we derive key  embeddings 
of  broken polynomial spaces into an appropriate Orlicz space.
We conclude with section~\ref{sec:computations} illustrating some computational results which demonstrate the robustness of the approximating scheme and the model 
when both the mesh discretization parameter and $\varepsilon$ vary.

\section{The Continuum Problem}
\label{sec:continuousProblem}

We assume the following bounds for the terms  in equation  (\ref{equ:total_potential}) 
\begin{equation}
\begin{aligned}
c_0 \left( |\bm 1 + \nabla \mbu|^2 -c_1 \right)  \le W(\nabla \mbu) 
\le c_2 \left( |\bm 1 + \nabla \mbu|^m + c_3 \right),
\end{aligned}
\label{W:properties}
\end{equation}
for some $m \ge 2$ and positive constants $c_0, c_1, c_2, c_3$.
 Also, we will assume that the penalty term satisfies the conditions
 \begin{subnumcases}{\Phi \ge 0, \quad \text{and} \quad \Phi(\nabla \mbu)  \le}
       &$C_0 \left( |\bm 1 + \nabla \mbu|^{2m_0} + C_1 \right)$  \label{pen:polynomial_growth}
       \\
        &$C_2e^{C_3 |\nabla \mbu|^2}$ \label{pen:exponential_growth}
\end{subnumcases}
again for some $m_0 \ge 1$ and positive constants $C_0, C_1, C_2, C_3$.   


%

 To define a minimization problem we should declare appropriate boundary conditions.
 We specify a globally injective, orientation preserving $\mbg \in  {H^3(\Omega)^2}$.    We encode boundary conditions in the following set:
\begin{equation}
\begin{aligned}
 \mathbb{A}(\Omega) = \{ \mbu\in H^2(\Omega)^2: \restr{\mbu}{\partial \Omega} = \restr{\mbg}{\partial \Omega}, \tred{  \restr{\nabla \mbu}{\partial \Omega} = \restr{\nabla \mbg}{\partial \Omega}}\}.
\end{aligned}
\label{problem_set}
\end{equation}
\tred{Notice that the boundary conditions were chosen for analytical convenience only. A variety of  other  boundary conditions can be treated with appropriate modifications in the finite element analysis, see Remark \ref{alternative_bc}. }
Now, the minimization problem can be defined as:
\begin{align}
\inf \{\Psi[\mbu]: \mbu \in \mathbb{A}(\Omega)   \}.
\label{continuous_minimuzation}
\end{align}


To ensure that the total potential energy has a minimizer $\mbu\in \mathbb{A}(\Omega)$,
one can prove that $\Psi[\cdot]$ is \textit{coercive} and \textit{lower semicontinuous}. The former 
can be derived from the properties of the strain energy function, i.e. (\ref{W:properties}),
 and the Poincar\'e inequality. One way to prove the latter
is to show convergence of the lower order term using $Vitali's$ Theorem and the convexity of 
the higher order term.  To avoid repeating similar proofs, these ideas will be used to show 
that an appropriate discretization
of the energy functional $\Gamma-$converge to the continuous total potential energy. 
Then, using a discrete compactness result we deduce that $\Psi[\cdot]$ admits a minimizer
using the \textit{Fundamental Theorem of $\Gamma-$convergence} 
\cite{braides2014global, braides2002gamma}.

%
%

\section{Discretization}
\label{sec:discretization}
\subsection{Notation}

\nomenclature{$T_h$}{The triangulation of the domain $\Omega$.}

Here we assume for simplicity that the domain $\Omega$ is polygonal and, henceforth,
$T_h$ denotes the triangulation of the domain $\Omega$ with mesh size $h$. For  $K \in T_h$, $K$ a triangle,
$h_K$ is the diameter of $K$ and the mesh size is then defined as $h := \max_{K \in T_h} h_K$.
The space of polynomials defined on $K$
 with total degree less than or equal to $q$ is denoted by $\mathbb{P}_q(K)$.
Next we will require the partitions of the domain to be \textit{shape regular} 
\cite{brenner2007mathematical}, i.e., there exists $c>0$ such that 
\begin{align}\label{shape_regular}
\rho_K \ge \frac{h_K}{c}, \quad  \text{for all } K \in T_h,
\end{align}
where $\rho_K$ is the diameter  of the largest ball inscribed in $K$.

The boundaries of the elements comprise the set of mesh edges $E_h$. 
The set $E_h$ is partitioned into $E^b_h$, the boundary, and 
$E^i_h$, the internal edges, such that $E^b_h= E_h \cap \partial \Omega$ and   
$E^i_h= E_h \setminus E^b_h$. For all $e \in E^i_h$ there exist two 
distinct elements, we denote them $K_{e^+}$ and $K_{e^-}$, such that
 $e \in \partial K_{e^+} \cap \partial K_{e^-}$.
Similarly, if  $e \in E^b_h$, there exists one element $K_e$ such that $e \in \partial K_e$.

For an edge $e \in \partial K$ 
we denote by $h_e$ its length. Assuming shape regularity 
it can be shown that there exist constants $C, c$ independent of the mesh size $h$, 
such that
\begin{align}
C h_K \le h_e \le c h_K, \quad \text{for } e \in \partial K \text{ and all } K \in T_h.
\label{equ:he_b_hK}
\end{align}

To discretize the continuous functional we need to first define our
finite element spaces. We use continuous and discontinuous families 
of Lagrange elements. Consider the space of continuous piecewise polynomial functions
 $V_h^q$, \textit{viz}.
\begin{equation}
\begin{aligned}
 V^q_h(\Omega) = \{ v \in C^0(\Omega):\restr{v}{K}  \in \mathbb{P}_q(K), 
 K \in T_h \}, \quad q \in \mathbb{N}.
 \label{def:continuous_space}
\end{aligned} 
\end{equation}
\nomenclature{$V^q_h(\Omega) $}{The space of continuous piecewise polynomials functions
of degree $\le q$, (\ref{def:continuous_space}).}
Also consider the discontinuous finite dimensional space 
\begin{align}
 \tV^k_h(\Omega) = \{ v \in L^2(\Omega): \restr{v}{K}  \in \mathbb{P}_k(K),
 K \in T_h\},  \quad k \in \mathbb{N}.
\label{def:dg_space}
\end{align}  
\nomenclature{$\tV^q_h(\Omega) $}{The space of discontinuous piecewise polynomials functions
of degree $\le q$, (\ref{def:dg_space}).}

We know that $V_h^q(\Omega) \subset H^1(\Omega)$. However, 
$V_h^q(\Omega)  \not \subset H^2(\Omega)$ and thus describing (\ref{equ:total_potential})
over $V_h^q(\Omega)$ will require the introduction
of penalty and jump terms in the discrete functional.  Notice that for $u_h \in V^q_h(\Omega)$
we have $\nabla u_h \in \tV^{q-1}_h(\Omega)$.
In the sequel, we shall use the following notation:
The \textit{trace} of 
functions in $\tV^q_h(\Omega)$ belong to the space
\begin{align}
 T(E_h) := \Pi_{e\in E_h} L^2(e),
\end{align}
where we recall that $E_h$ is the set of mesh edges.
The \textit{average} and \textit{jump} operators over
$T(E_h)$ for $\mbw \in T(E_h)^{2 \times 2 \times 2}$ and 
$\mbv \in T(E_h)^{2 \times 2}$ are defined by:
\begin{equation}
 \begin{aligned}
\dgal{ \cdot } :& T(E_h)^{2 \times 2 \times 2} \mapsto L(E_h)^{2 \times 2 \times 2} \\
  &\dgal{ \mbw } := \left\{
        \begin{array}{ll}
             \frac{1}{2} (\restr{\mbw}{K_{e^+} }  + \restr{\mbw}{K_{e^-} }),
             & \quad \text{for } e \in E_h^i \\
             \mbw, & \quad \text{for } e \in E_h^b,
        \end{array}
    \right.
 \end{aligned}
\label{average_operator}
\end{equation}
\nomenclature{$\dgal{ \cdot } $}{The average operator,  (\ref{average_operator}).}
\begin{equation}
 \begin{aligned}
 \jumpop{ \cdot }
 :& T(E_h)^{2 \times 2} \mapsto L(E_h)^{2 \times 2} \\
  &\jumpop{ \mbv }	 :=
             \restr{\mbv}{K_{e^+} }  - \restr{\mbv}{K_{e^-} },
             & \quad \text{for } e \in E_h^i \\
 \end{aligned}
\label{jump_operator}
\end{equation}
\begin{equation}
 \begin{aligned}
 \jumpop{ \cdot }
 :& T(E_h)^{2 \times 2 \times 2} \mapsto L(E_h)^{2 \times 2 \times 2} \\
  &\jumpop{ \mbv \otimes \mbn_{e}}	 := 
              \restr{\mbv}{K_{e^+} } \otimes \mbn_{e^+}  + \restr{\mbv}{K_{e^-} }\otimes \mbn_{e^-},
             & \quad \text{for } e \in E_h^i \\
 \end{aligned}
\label{tensor_jump_operator}
\end{equation}
where $K_{e^+}$, $K_{e^-}$ are the  elements that share the internal edge $e$; 
$\mbn_{e^+}, \mbn_{e^-}$ are the corresponding outward normal to the edge 
and $\mbv \otimes \mbn_{e}$ is a third order tensor with 
$\left( \mbv \otimes \mbn_{e} \right)_{ijk} = \mbv_{ij} \mbn_{e_k}$. \tred{Since the boundary terms encode implicitly the boundary conditions
we prefer to define them below, see \eqref{boundary_faces}}.
 {We employ the usual summation convention, e.g. $\mbv \cdot \mbv = v_{ij} v_{ij}$; also 
subscripts preceded by a comma indicate partial differentiation with the respect to the 
corresponding coordinate, e.g. $ f_{i,j}=\partial f_i=\partial x_j$.}

\subsection{Discretization of the Energy functional}

A direct discretization of the minimization problem (\ref{equ:total_potential}) 
would require an approximation space, a subspace of $H^2(\Omega)\times H^2(\Omega)$.
This means that, for conforming finite elements, we would require $C^1$ continuity at the 
interfaces, i.e. across element internal boundaries. It is well known that
 the construction of elements that ensure $C^1$ continuity is quite complex.
Here we adopt to our problem 
an alternative approach based on the discontinuous Galerkin formulation. Our approximations
 will be sought on $ V_h^q(\Omega)^2;$ however the energy functional 
should be modified to account for possible discontinuities of normal derivatives at the element faces. 
The appropriate modification of the energy functional proposed
 below is motivated by the analysis in  
\cite{makridakis2014atomistic}; the resulting bilinear form of the biharmonic operator
 obtained via the  first variation, will be the form of the $C^0$ discontinuous Galerkin method
  for the linear biharmonic problem,  introduced in \cite{engel2002continuous}. 


The discretized functional for  $\mbu_h \in V^q_h(\Omega)^2$, $q \geq 2$,
 has the form:
\begin{equation}
\begin{aligned}
\Psi_h[\mbu_h] &= \int_{\Omega} [W(\nabla \mbu_h) +
\Phi(\nabla \mbu_h)]  
\\ & + \varepsilon^2   \Bigg(
\frac{1}{2}\sum_{K \in T_h}  \int_{K} | \nabla \nabla \mbu_h |^2 
-\sum_{\tred{e \in E_h}} \Big[ 
\int_{e}\dgal{ \nabla \nabla \mbu_h} \cdot \jumpop{\nabla \mbu_h
\otimes \mbn_e} + \frac{\alpha}{h_e}\int_{e} |\jumpop{ \nabla 
\mbu_{h}}|^2  \Big] \Bigg) \\
&=\int_{\Omega} [W(\nabla \mbu_h) +
\Phi(\nabla \mbu_h)]  + \varepsilon^2 \Psi^{ho}_h[\mbu_h] 
\end{aligned}
\label{equ:final_potential}
\end{equation}
 where the functional
$ \Psi^{ho}_h[\mbu_h] $ contains the higher order terms.
 \tred{For the boundary faces  we use the notation:}
\begin{equation}\label{boundary_faces}
\jumpop{ \nabla 
\mbu_{h}  }:=  \nabla 
\mbu_{h}^-
 -\nabla 
I_h^q \mbg \, , \text { and } \jumpop{\nabla 
\mbu_{h}  \otimes \mbn_{e}} :=	(\nabla 
\mbu_{h}^-
 -\nabla 
I_h^q \mbg)  \otimes \mbn, \text { for } e\in E^b_h= E_h \cap \partial \Omega \, ,
\end{equation}
\tred{where $\mbg$ is given in (\ref{problem_set}),  $I^q_h$ is the standard nodal interpolation operator in \eqref{Interp_operator} and $\mbn$ the outward normal on $\partial \Omega.$ 
 {Note that the stabilization term is independent of $\mbn_e$ because $|\nabla \mbu_h \otimes \mbn_e|^2 = |\nabla \mbu_h|^2$.} 
Although the   boundary condition on $\nabla \mbg$ is encoded implicitly in the discrete functional through \eqref{boundary_faces},  the boundary condition on $\mbg$ is enforced explicitly in the discrete space:}
\begin{equation}
\begin{aligned}
 \mathbb{A}_h^q(\Omega) = \{ \mbu_h \in V^q_h(\Omega)^2: 
\restr{\mbu_h}{\partial \Omega} =\restr{\mbg_h}{\partial \Omega} \},
\end{aligned}
\label{def:Aqh}
\end{equation}
where $\mbg_h = I^q_h \mbg. $ \tred {From now on we shall use the convention that for elements of $\mathbb{A}_h^q(\Omega)$ the jumps on the boundary faces are given through
\eqref{boundary_faces}.}
So, we have to solve the corresponding discrete minimization problem
\begin{align}
\inf \{\Psi_h[\mbu_h]: \mbu_h \in \mathbb{A}^q_h(\Omega)   \}.
\end{align}

Clearly $\nabla \mbu_h \in \tV^{q-1}_h(\Omega)^{2 \times 2}$.
On the other hand, $\nabla \nabla \mbu_h$ does not exist as a function in $L^2(\Omega)$ 
 and it can be defined only in the piecewise sense 
 at the element level, i.e. $\restr{\nabla \nabla \mbu_h}{K} \in \mathbb{P}_{q-2}(K)$.
 
 \tred{
 \begin{Remark} [Alternative boundary conditions]
 \label{alternative_bc}
Modeling and simulations of experimental results do not require strict boundary conditions of the 
function $\mbu$ and its derivative.
It is observed in the experiments, that when cells contract and pull the ECM fibers, they deform inhomogeneously and change shape under inhomogeneous forces. A minimal model
for this deformation has been employed in \cite{Grekas_2021}, where the boundary of the cell is 
connected to the ECM
with linear springs.  This contributes to the continuous model energy the term
\begin{align}
\frac{k}{2} \int_{\Gamma_i} |\mbu - \mbg |^2 ds,  \text{ where } \Gamma_i \subset \partial \Omega
\text{ and } | \Gamma_i | >0, 
\label{eq:linear_springs}
\end{align} 
where $k$ is the  stiffness constant and $\Gamma_i$ are the boundaries of the cells. See \cite[5.2.2. 
Model for active particles]{Grekas_2021} for more details. 
For large enough values of $k$, the above term can model Dirichlet type boundary conditions,
in the sense that the displacement is imposed by a stiff linear spring, which is perhaps closer to the 
experimental approach than the explicit enforcement of the boundary 
conditions. Then, our analysis remains valid, where now $\mbu \in H^2(\Omega)^2$,
 $\mbu_h \in V^q_h(\Omega)^2$ and the terms involving boundary edges  
 in (\ref{equ:final_potential}) and elsewhere are excluded.
 \end{Remark}
 }




%

\section{Preliminary results}

\subsection{Preliminary results for finite element spaces}

For convenience we briefly state some preliminary results on the
finite element spaces which will be useful in the sequel.
Following partially the notation of Brenner \& Scott, 
\cite{brenner2007mathematical}, let $\hat{K} = \{(1/h_K)\mbx: 
\mbx\in K \}$ and, for $w \in \mathbb{P}_q(K)$, define the 
function $\hat{w} \in \mathbb{P}_q(\hat{K} )$ by $ \hat{w}(\hat{\mbx}) = w(h_K\hat{\mbx})$.
Then $w \in W^{p, r}(K)$ is equivalent to $\hat{w} \in W^{p, r}(\hat{K})$
and
\begin{align}
 |\hat{w}|_{W^{p,r}(\hat{K})} = h_K^{p - n/r}|w|_{W^{p,r}(K)}.
 \label{eq:scaled_K}
\end{align}

Next we state the well known trace inequality:
\begin{Lemma}
\label{Theorem:Trace_inequality}
 Assume that $\Omega$ is bounded and has a Lipschitz boundary. Let $w \in 
W^{1,p}(\Omega)$, $p \in [1, +\infty]$. Then there exists a constant 
$C$ depending 
only on $p$ and $\Omega$ such that
\begin{align}
 \norm{w}_{L^p(\partial \Omega)} \le C \norm{w}_{L^p( 
\Omega)}^{1-1/p} \norm{w}_{W^{1,p}(\Omega)}^{1/p}.
\label{eq:trace_inequality}
\end{align}
\end{Lemma}
In general one can have an estimate of the above constant, for 
instance if $\Omega$ is the unit disk in $\mathbb{R}^2$ and $p=2$ then $C \le 
8^{1/4}$, see \cite{brenner2007mathematical}. We  state the discrete trace 
inequality which is a consequence of Lemma~\ref{Theorem:Trace_inequality} and of 
(\ref{eq:scaled_K}).

\begin{Lemma}[Discrete Trace Inequality]
Let $T_h$ be a shape regular triangulation. Then, there exists a constant $c_q$ independent of 
$h$, but depending on $q$, such that
\begin{align}
 \norm{ u_h}^2_{L^2(e)} \leq \frac{c_q}{h_e} \norm{u_h}^2_K, \quad \forall u_h 
\in \mathbb{P}_q(K), \forall K 
 \in T_h .
\label{eq:trace_ineq}
 \end{align}

\end{Lemma}

The standard nodal interpolation operator 
 will be denoted by $I_h^q$, where   
\begin{align}
I_h^q :  {H^s(\Omega)^2} \rightarrow V_h^q(\Omega)^2, \quad  {s \ge 2.}
\label{Interp_operator}
\end{align}

\nomenclature{$I_h^q: $}{The nodal interpolation operator, (\ref{Interp_operator}).}
Next some well known interpolation error estimates are presented.

\begin{Lemma}
\label{Lemma:error_on_K}
Let $\mbu \in H^s(\Omega)^2$, with $s \ge 2$, $T_h$ be a 
shape-regular triangulation of the domain $\Omega$.
 If $q \ge \left \lceil{s}\right \rceil -1$, where $\left \lceil{s}\right \rceil$ is the smallest integer value greater than or equal to $s$,
 then there exists a constant $c$ depending only on the domain $\Omega$,
 a shape parameter of the triangulation and $s$ such that 
 \begin{align}
 | \mbu - I_h^{q}\mbu|_{H^m(K)} \le c h_K^{s-m} | \mbu|_{H^s(K)},
  \quad 0 \le m \le s
  \label{error_est}
 \end{align}
 where $h_K$ denotes the diameter of the element $K$.
\end{Lemma}

Using the trace inequality (\ref{eq:trace_inequality})  we obtain the error estimates for
norms defined on the mesh edges. 
\begin{Corollary}\label{est_e_K}
If the assumptions of Lemma \ref{Lemma:error_on_K} hold, then  for a face $e$ 
of an element $K$, i.e. $e \in \partial K$ we have the following error estimate of the integral 
over $e$:
 \begin{align}
 | \mbu - I_h^{q}\mbu|_{H^m(e)} \le c h_K^{s-m -1/2} | \mbu|_{H^s(K)},
\label{ineq:error_on_e}
\end{align}

%

\end{Corollary}

\subsection{Poincar\'e Inequalities for Broken Sobolev Spaces}


To bound $\mbv_h\in V^q_h(\Omega)^2$ in higher order norms
we will need  Poincar\'e inequalities for broken Sobolev spaces.
For this purpose  we define the broken Sobolev seminorm for \tred{$\mbw \in  \mathbb{A}_h^q(\Omega) $:}
\begin{equation}
\begin{aligned}
 |\mbw|_{H^2(\Omega, T_h)}^2  :=  \sum_{K \in T_h}\int_K |\nabla \nabla \mbw|^2 + \sum_{e \in E_h}
 \frac{1}{h_e}\int_e |\jumpop{\nabla \mbw}|^2.
\end{aligned}
\label{eq:seminorm}
 \end{equation}
 
Since $\Omega \subset \mathbb{R}^2$ is an open, connected  and bounded set with Lipschitz
boundary, the main result in 
\cite{lasis2003poincare}
and in \cite{brenner2003poincare} implies that
\begin{equation}
\begin{aligned}
  \norm{\nabla \mbw}_{L^r(\Omega)}^2 \le c_{\theta_{min} }\Big(  
  |\mbw|_{H^2(\Omega, T_h) } ^2+   |\widehat{\nabla \mbw}|^2
  \Big)
\end{aligned}
\label{Poincare_ineq}
\end{equation}
for all $r \in [1, +\infty)$, where  $\widehat{\nabla \mbw}$ can take any
of the following forms
\begin{equation}
\begin{aligned}
\widehat{\nabla \mbw} =  \left\{
        \begin{array}{ll}
           \frac{1}{|\Omega| } \int_\Omega \nabla \mbw \\
            \frac{1}{|\partial \Omega|}\int_{\partial \Omega} \nabla  \mbw \\
            \frac{1}{|\partial \Gamma_i|}\int_{\partial \Gamma_i} \nabla  \mbw, \quad \Gamma_i 
            \subset \partial \Omega \text{ and } |\Gamma_i|> 0.
        \end{array}
    \right.
\end{aligned}
\end{equation}
The constant $c_{\theta_{min} }$ depends on the minimum angle of the triangles,  $\theta_{min}$, 
and $r$. We have assumed that the family of partitions $\{ T_h \}$ is shape regular, consequently
$\theta_{min}$ is independent of $h$, see \cite{lasis2003poincare} for details.

\subsection{Lifting and the Discrete Gradient Operators}

We adopt in our case of the discontinuous gradient the results of  \cite{bassi1997high, brezzi2000discontinuous}
regarding Lifting and Discrete gradients.

\begin{Definition}[The vectorial piecewise gradient]
\label{def:vec_pg} 
Let $\mbw \in \tV^k_h(\Omega)^m $, $\Omega \subset \mathbb{R}^n$. The vectorial piecewise gradient 
$\nabla_h: \tV^k_h(\Omega)^m \rightarrow \tV^{k-1}_h(\Omega)^{m \times n}$
is defined componentwise as
\begin{align}
\restr{(\nabla_h w_i)}{K} :=  \nabla (\restr{w_i}{K} ), \quad  i=1, ..., 
m, \forall K \in T_h.
\label{vec_piecewise_gradient}
\end{align}

Consequently for all $ \mbu_h \in V_h^q(\Omega)^2$, $\nabla_h 
\nabla \mbu_h \in L^2(\Omega)^{2 \times 2 \times 2}$ and 
\begin{align}
\int_\Omega \nabla_h \nabla \mbu_h d\mbx= \sum_{K \in T_h}  \int_{K} | \nabla \nabla \mbu_h |^2 d\mbx .
\label{eq:piecewise_grad}
\end{align}
\end{Definition}
\nomenclature{$\nabla_h: $}{ vectorial piecewise gradient, (\ref{eq:piecewise_grad}).}

For all $e \in E_h$ we define the linear operator, known as lifting operator 
\cite{bassi1997high, brezzi2000discontinuous},
$\mbr_e:L^2(e)^{2 \times 2} \rightarrow \tV^{q-2}_h(\Omega)^{2 \times 2 
\times 2}$, for all $\mbph \in L^2(e)^{2 \times 2}$ as 
follows
\begin{align}
 \int_\Omega \mbr_e(\mbph) \cdot \mbw_h d\mbx = \int_e \dgal{ \mbw_h } \cdot 
\jumpop{ \mbph \otimes \mbn_e } ds, \quad  \forall \mbw_h \in \tV^{q-2}_h(\Omega)^{2 \times 2 
\times 2}.
\label{eq:lift_operator}
\end{align}
It can be shown that $\mbr_e(\phi)$ is non zero only on the elements that 
contain $e$ on their boundary, i.e. $supp(\mbr_e) =  \{K \in T_h : e \in 
\partial K\}$.
We also define the global lifting operator
$$ \mbR_h (\phi) = \sum_{\tred{e \in E_h}} \mbr_e (\phi).$$
With the help of this operator we can represent the second term 
 {of the functional $\Phi^{ho}_h$,
see (\ref{equ:final_potential})}, as an integral over $\Omega$; namely,
\begin{align}
\int_\Omega \mbR_h( \nabla \mbu_h ) \cdot \mbw_h d\mbx = \sum_{\tred{e \in E_h}}\int_e 
\dgal{ \mbw_h } \cdot 
    \jumpop{\nabla \mbu_h \otimes \mbn_e} ds,
\label{eq:R_applied}
\end{align} 
Since $\mbu_h \in  \mathbb{A}_h^q(\Omega)$, we can substitute $\mbw_h$ by 
$\nabla_h \nabla \mbu_h$ to obtain 
\begin{equation}
\begin{aligned}
\sum_{\tred{e \in E_h} }\int_{e} \dgal{ \nabla \nabla \mbu_h } \cdot \jumpop{\nabla \mbu_h
\otimes \mbn_e} =
 \sum_{\tred{e \in E_h} }\int_{e} \dgal{ \nabla_h \nabla \mbu_h } \cdot \jumpop{\nabla \mbu_h
\otimes \mbn_e}
= \int_\Omega \mbR_h( \nabla \mbu_h ) \cdot \nabla_h \nabla \mbu_h,
\end{aligned}
\label{eq:rel2lift}
\end{equation} 
\tred{where  recall that at the boundary faces $e\in E^b_h$ we use the convention  
$\jumpop{ \nabla 
\mbu_{h}  }=  \nabla 
\mbu_{h}^-
 -\nabla 
I_h^q \mbg \, , $ and $\dgal{ \nabla \nabla \mbu_h} = \nabla \nabla \mbu_h^-\, .$}

We next define the discrete gradient $G_h$, which is a combination of the vectorial 
piecewise gradient and the global lifting operator.
In particular, the discrete gradient of $\nabla \mbu_h$ is defined as
\begin{align}
G_h(\nabla \mbu_h) = \nabla_h \nabla \mbu_h - \mbR_h( \nabla \mbu_h ).
\label{discrete_gradient}
\end{align}
 Later we will show under which conditions $G_h(\nabla \mbu_h)  \rightharpoonup \nabla 
\nabla \mbu$ in  $L^2(\Omega)^{2 \times 2 \times 2}$, when $\mbu_h \rightarrow \mbu$ 
in $H^1(\Omega)^2$ as $h \rightarrow 0$. For this reason it will be convenient
 to write the higher order derivatives of the  discrete
total potential energy in terms of the discrete gradient and the global lifting operator.
We do this in Lemma \ref{prop:ho2disc_gradient} below which  will 
be useful in the sequel.
\begin{Lemma}
\label{prop:ho2disc_gradient}
The higher order terms of the discretized total potential energy can be written
 in terms of the discrete gradient and the global lifting operator, as follows 
\begin{equation}
\begin{aligned}
 \frac{1}{2}\sum_{K \in T_h}  \int_{K} | \nabla \nabla \mbu_h |^2  
-\sum_{\tred{e \in E_h}}\int_{e}\dgal{ \nabla \nabla \mbu_h } \cdot 
\jumpop{\nabla \mbu_h \otimes \mbn_e}
=
 \frac{1}{2} \int_{\Omega} | G_h(\nabla \mbu_h) |^2 -  |\mbR_h( \nabla \mbu_h ) |^2.
\end{aligned}
\end{equation}

\begin{proof}
The proof follows from Definition \ref{def:vec_pg}  and equation (\ref{eq:rel2lift}).
\end{proof}
\end{Lemma}
From Lemma \ref{prop:ho2disc_gradient} the functional $\Psi^{ho}_h$, 
 {displayed in (\ref{equ:final_potential})}, becomes
\begin{equation}
\begin{aligned}
\Psi^{ho}_h[\mbu_h] =&  \frac{1}{2} \int_{\Omega} | G_h(\nabla \mbu_h) |^2 - 
 |\mbR_h( \nabla \mbu_h) |^2 d\mbx 
 +  {   \sum_{\tred{e \in E_h}} \frac{\alpha}{h_e}\int_{e} |\jumpop{\nabla \mbu_h}|^2  ds. }
\end{aligned}
\label{eq:hoPsi2_discrete_gradient}
\end{equation}
Using inverse inequalities (\ref{eq:trace_ineq}), one can show,  see \cite{brezzi2000discontinuous}, that 
 there exists a positive constant $C$, independent of h, such that
 \begin{align}
  \norm{\mbr_e( \nabla \mbu_h )}_{L^2(\Omega)^{2 \times 2 \times 2}} \leq 
  C_r h_e^{-1/2} \norm{  \jumpop{\nabla \mbu_h} }_{L^2(e)^{2 \times 2} }
  \label{eq:r_e_bound}
 \end{align}
This bound finally implies the next Lemma, see \cite[{Lemma 4.34}]{di2011mathematical}
 and \cite[Lemma 7]{buffa2009compact} for details.
  \begin{Lemma}[Bound on global lifting operator] 
  \label{prop:Rh_bound}
For all \tred{$\mbu_h \in  \mathbb{A}_h^q(\Omega)$} there holds 
  \begin{align}
   \int_\Omega |\mbR_h( \nabla \mbu_h)|^2  \leq 
  C_R \sum_{\tred{e \in E_h}}h_e^{-1} \int_e |\jumpop{\nabla \mbu_h}|^2 ds,
  \label{eq:R_bound}
  \end{align}
  where the constant $C_R$ depends on the constant of (\ref{eq:r_e_bound}).
\end{Lemma}

\begin{Corollary}[Bound on Discrete Gradient]
There exists a constant $C>0$ such that for all  \tred{$\mbu_h \in  \mathbb{A}_h^q(\Omega)$}
 it holds that
\begin{align}
\norm{G_h (\nabla \mbu_h)}_{L^2(\Omega)^{2 \times 2 \times 2}} 
\le C |\mbu_h|_{H^2(\Omega, T_h)}.
\label{eq:discrete_gradient_bound}
\end{align} 

\end{Corollary}

\subsection{Analytical preliminaries} 
In the subsequent sections we examine the convergence 
of the lower order terms  in $L^1(\Omega)$,  i.e. the terms $W(\cdot)$ and $\Phi(\cdot)$
of (\ref{equ:final_potential}), when $\mbu_h \rightarrow \mbu$ in $H^1(\Omega)$.    
For this purpose we employ Vitali's convergence theorem \cite{royden1968real}.

\begin{Theorem}[Vitali convergence theorem] 
\label{Theorem:Vitali}
Let $\Omega$ be a set of finite measure and $(f_n)$ be a sequence of functions in $L^1(\Omega)$. 
Assume $(f_n)$ is uniformly integrable over $\Omega$, i.e., for each $\epsilon > 0$, 
there exist $\delta >0$ independent of $n$ such that
$\text{if } A \subset \Omega \text{ measurable  and } |A|  < \delta,
\text{ then } \int_A f_n d\mbx < \epsilon.$
If $(f_n) \rightarrow f$ pointwise a.e. on $\Omega$, then $f \in L^1(\Omega)$ and 
\begin{align}
\lim_{n \rightarrow +\infty}\int_\Omega f_n(\mbx) d \mbx = \int_\Omega f(\mbx) d \mbx.
\end{align}
\end{Theorem}
We now state a useful criterion 	of uniform integrability,  {known} as the
\textit{de la Vall\'ee Poussin} criterion, 
\cite[Theorem 4.5.9]{bogachev2007measure}.

\begin{Theorem}[de la Vall\'ee Poussin criterion]
A family $(f_n) \subset L^1(\Omega)$, is uniformly integrable if, and only if, there exists a
non-negative increasing function $G$  on $[0, +\infty)$ such that 
\begin{equation}
\lim_{t \rightarrow +\infty} \frac{G(t)}{t} = +\infty \quad \text{and} \quad \sup_n \int_\Omega G \left( |f_n(\mbx)| \right) d \mbx < +\infty. 
\end{equation}

\end{Theorem}

\section{$\Gamma-$Convergence of the discretization}
\label{sec:G_convergence}

In this section we establish the $\Gamma$-convergence of the discretized functionals $\Psi_h$ to the continuum energy $\Psi$. 
The proof consists of three parts: we first prove equi-coercivity, a necessary property to bound  
the sequence $(\mbu_h)$ when the family of discrete energies $(\Psi_h[\mbu_h])$ is bounded.
Then, we show
the $\liminf$ inequality which provides a lower bound of the discrete energies by the continuum counterpart. 
We conclude with the $\limsup$ inequality which, as we will see in Section \ref{sec:compactness_convergence} ensures the attainment of the limit. 

\subsection{Equi-Coercivity and Convergence of the Discrete Gradient}

\begin{prop}[Equi-coercivity]
\label{equi-coercivity}
Assume that $ {\alpha} >C_R$, i.e.  {the penalty parameter 
in (\ref{equ:final_potential})} is greater than the constant of Proposition 
\ref{prop:Rh_bound}.
Let $(\mbu_h)_{h>0}$ be a sequence of displacements in \tred{$\mathbb{A}^q_h(\Omega)$}
 such that for a constant $C>0$ independent of $h$ it holds that
$$
\Psi_h[\mbu_h] \leq C.
$$  
Then the there exists a constant $C_1>0$ such that
\begin{align}
&|\mbu_h|^2_{H^2(\Omega, T_h)} \le C_1 . \label{coer:dg_seminorm}
\end{align}
 In addition, 
 \begin{align}
 \norm{\mbu_h}_{H^1(\Omega)^2} \le C_2, \label{u_H1norm}
\end{align}
for a positive constant $C_2$, where  $C_1, C_2$ are independent of $h$.
\begin{proof}
We have shown,  {see (\ref{equ:final_potential})} and (\ref{eq:hoPsi2_discrete_gradient}),  
that $\Psi^{ho}[\mbu_h]$ can be
 written in terms of the discrete gradient $G_h$ and the lifting operator $\mbR_h$.
From the assumption $a > C_R$ and the bound of the global lifting operator in (\ref{eq:R_bound}), 
we see that $\Psi^{ho}[\mbu_h]$ is nonnegative:
\begin{equation}
\begin{aligned}
\Psi^{ho}_h[\mbu_h] &= \frac{1}{2} \int_{\Omega} | G_h(\nabla \mbu_h)|^2 
 - \int_{\Omega} \mbR_h( \nabla \mbu_h ) \cdot \mbR_h( \nabla \mbu_h )
+ 
   {\sum_{e \in E_h } \frac{\alpha}{h_e}\int_{e} |\jumpop{\nabla \mbu_h}|^2  }
\\ &\ge
 \frac{1}{2} \int_{\Omega} | G_h(\nabla \mbu_h)|^2 
+ \sum_{e \in E_h } \frac{\alpha -C_R}{h_e}\int_{e} |\jumpop{\nabla \mbu_h}|^2 \ge 0.
\end{aligned}
\label{Psi_ho_ubound}
\end{equation}
In particular,
\begin{align}
\Psi^{ho}_h[\mbu_h] \ge  \frac{1}{2} \int_{\Omega} | G_h(\nabla \mbu_h)|^2
\quad \text{and} \quad 
 {\Psi^{ho}_h[\mbu_h] \ge \sum_{e \in E_h^i} \frac{\alpha -C_R}{h_e}\int_{e} |\jumpop{\nabla \mbu_h}|^2.} \label{eq:Psi_ho_b1}
\end{align}

By (\ref{W:properties}) it holds $W(\nabla \mbu_h) \ge -c$, $c$ is a positive constant,  
and $\Phi(\nabla \mbu_h)$ is a nonnegative penalty parameter,
consequently
\begin{equation}
\begin{aligned}
\Psi_h[\mbu_h] &\ge \int_{\Omega} W(\nabla \mbu_h(\mbx))
\quad \text{and} \quad c + \Psi_h[\mbu_h] \ge  \varepsilon^2 \Psi^{ho}_h[\mbu_h].  
\end{aligned}
\label{Psi_ge_toall}
\end{equation}
From the assumption that $\Psi_h[\mbu_h]$ is uniformly bounded over $h$, i.e. $\Psi_h[\mbu_h] \le C$
 for all $ h >0$, we obtain that all terms  {appearing in the right hand sides of (\ref{eq:Psi_ho_b1})  and (\ref{Psi_ge_toall})} are uniformly bounded.
 Therefore, using the  
bound of the global lifting operator, (\ref{eq:R_bound}), we conclude that 
\begin{equation}
\begin{aligned}
\int_\Omega |\nabla_h \nabla \mbu_h|^2 &\le 2\int_\Omega |\nabla_h \nabla \mbu_h -
 \mbR_h( \nabla \mbu_h )|^2 + 2\int_\Omega \mbR_h( \nabla \mbu_h )|^2 
 \\ & \le
 2\int_\Omega | G_h(\nabla \mbu_h)|^2 + 2  C_R \sum_{e \in E_h}h_e^{-1} \int_e |\jumpop{\nabla \mbu_h}|^2
 {\le C^{\prime} }
\end{aligned}
\end{equation}
and 
\begin{equation}
\begin{aligned}
|\mbu_h|^2_{H^2(\Omega, T_h)} &= \int_\Omega |\nabla_h \nabla \mbu_h|^2 + \sum_{e \in E_h}h_e^{-1} \int_e |\jumpop{\nabla \mbu_h}|^2 
\le C_1.
 \end{aligned}
 \end{equation}
It remains to show that  $\norm{\mbu_h}_{H^1(\Omega)^2}$ is uniformly bounded, 
using that  $\mbu_h \in \mathbb{A}_h^q(\Omega)$.  
By the coercivity condition on $W$ given in (\ref{W:properties}), equation (\ref{Psi_ge_toall}) gives
\begin{multline}
C \ge \Psi[\mbu_h] \ge \int_{\Omega} W(\nabla \mbu_h(\mbx)) \ge 
 c_0\int_{\Omega} \left( |\bm{1} + \nabla \mbu_h|^2 - c_1 \right) 
  \ge c_0\int_{\Omega} \left( |\bm{1}|^2  + |\nabla \mbu_h|^2 - 2|\bm{1}||\nabla \mbu_h| - c_1  \right) 
 \\
 \ge c_0\int_{\Omega} \left( |\bm{1}|^2  + |\nabla \mbu_h|^2 - \frac{|\bm{1}|^2}{\delta} - \delta|\nabla \mbu_h|^2 - c_1  \right),
\end{multline}
where we have used the Cauchy-Schwarz and Young's inequality. Choosing, for example, $\delta=1/2$, we infer that
$
C \ge \frac{c_0}{2} \int_{\Omega} \left( |\nabla \mbu_h|^2- c \right) d\mbx 
$
and the proof is concluded by Poincar\'e's inequality.
\end{proof}
\end{prop}

\subsection{The $\liminf$ inequality}

\begin{Lemma}[Convergence of the lower order terms]
\label{Lemma:lo_convergence0}
Let $\mbu_h \rightarrow \mbu$ in $H^1(\Omega)^2$, with $\mbu_h \in H^1(\Omega)^2$ 
and $\mbu \in H^2(\Omega)^2$. Suppose further that
\begin{align}
\norm{\nabla \mbu_h}_{L^r(\Omega)^{2 \times 2}} < C
 \quad \text{for all } r \in [1, +\infty) \text{ and for all } h >0, 
\label{unif_bound}
\end{align}
where $C$ is independent of $h$. Then, up to a subsequence,
\begin{align}
 \int_\Omega W(\nabla \mbu_h) \rightarrow \int_\Omega W(\nabla \mbu). \label{W:convergence0}
\end{align} 
\end{Lemma}
\begin{proof}
From the assumption $\mbu_h \rightarrow \mbu$ in $H^1(\Omega)^2$ there exists a
subsequence, not relabeled, such that $\nabla \mbu_h \rightarrow \nabla \mbu$ a.e.
Note that $W$, see  (\ref{W:properties}), is bounded from above. Specifically 
\begin{align}
0 \le W(\nabla \mbu)  = \hat{W}(\mbF)  \le c_1 \left( |\mbF|^m + 1 \right),
\end{align} 
for some $m \in \mathbb{N}$.
Since $(\nabla \mbu_h)$ is bounded in all $L^r$ norms, by (\ref{unif_bound}), this implies
 that $(W(\nabla u_h))$  is uniformly integrable and from
Vitali's Theorem~\ref{Theorem:Vitali} must converge to $W(\nabla \mbu)$ in $L^1(\Omega)$.
\end{proof}

\begin{Remark}[Convergence of the penalty term]
\label{Remark:penalty_convergence}
From inequality (\ref{pen:polynomial_growth}) the penalty term is bounded from above, i.e.
\[   
\Phi(\nabla \mbu) \le c_0 \left| \mbF \right|^{m} + C_1.
\] 
Then as in Lemma~\ref{Lemma:lo_convergence0} 
$\Phi(\nabla \mbu_h) \rightarrow \Phi(\nabla \mbu)$, as $h \rightarrow 0$ up to a subsequence.
\end{Remark}

  \begin{Lemma}
  \label{Lemma:Weak_D_gradient_weak_limit}
  Let $\mbu_h \rightarrow \mbu$ in 
$H^1(\Omega)^2$,  with \tred{$\mbu_h\in \mathbb{A}^q_h(\Omega)$}. If $|\mbu_h|_{H^2(\Omega, T_h)}$
  is uniformly bounded with respect to  $h$, then 
  
  \begin{equation}
\begin{aligned}
\lim_{h \rightarrow 0 }\int_\Omega G_h(\nabla \mbu_h) \cdot \mbph =
 -\int_\Omega \nabla \mbu \cdot (\nabla \cdot \mbph), \quad 
 \forall  \mbph \in  C_c^\infty(\Omega)^{2 \times 2 \times 2},  \text{ where $(\nabla \cdot \mbph)_{ij} = \mbph_{ijk,k}$}.
\end{aligned}
\label{eq:discrete_gradient_distr_convergence}
\end{equation}

\end{Lemma}

\begin{proof}
This proof is an adaptation of the proof in \cite[Theorem 2.2]{di2010discrete}. 
Let $\mbph \in C_c^\infty(\Omega)^{2 \times 2 \times 2}$, then 
from the definition of the piecewise gradient, equation (\ref{vec_piecewise_gradient}),
  and the divergence theorem, we obtain
\begin{equation}
\begin{aligned}
 \int_\Omega &G_h(\nabla \mbu_h) \cdot \mbph  =
 \sum_{K \in T_h} \int_K\nabla \nabla \mbu_h \cdot \mbph - \int_\Omega 
\mbR_h( \nabla \mbu_h ) \cdot \mbph  
\\
=& -\int_\Omega \nabla \mbu_h \cdot (\nabla \cdot \mbph) 
+  \sum_{e \in E _h^i}\int_{e} \mbph \cdot \jumpop{\nabla \mbu_h \otimes \mbn_e }
-\int_\Omega 
\mbR_h( \nabla \mbu_h) \cdot \mbph ,
\end{aligned}
\label{eq:weak_G_conv}
\end{equation}
where we recall that  {$(\nabla\mbu_h\otimes \mbn)_{ijk} = (\mbu_h)_{i,j}\mbn_k$}.

First, we show that the last two terms convergence to $0$ as $h 
\rightarrow 0$. To this end, 
let $\bar{I}^0_h$ be the piecewise average operator onto
$\tV^0_h(\Omega)^{2 \times 2\times 2}$, i.e. 
$\restr{\bar{I}^0_h \mbph}{K}  = 1/|K| \int_K \mbphi$,
 and define $\mbph_h =\bar{I}^0_h \mbph$. 
Then from standard error estimates,  see e.g. \cite[Lemma 1.58]{di2011mathematical},
\begin{equation}
\begin{aligned}
\int_\Omega |\mbph - \mbph_h|^2 & 
\le \sum_{K \in T_h} c_K h^2_K \int_K |\nabla \mbph |^2  \le c h^2 \int_\Omega |\nabla \mbph |^2 \rightarrow 0.
\end{aligned}
\label{eq:phi_h_conv}
\end{equation}
Now, using the definition of $\mbR_h$, see (\ref{eq:R_applied}), for the last two terms of 
(\ref{eq:weak_G_conv}) we obtain
\begin{equation}
 \begin{aligned}
 &\sum_{e \in E  _h^i}\int_{e} \mbph \cdot \jumpop{\nabla \mbu_h \otimes \mbn_e }- 
\int_\Omega 
\mbR_h( \nabla \mbu_h ) \cdot \mbph
\\ 
=& \sum_{e \in E _h}\int_{e} \dgal{ \mbph -\mbph_h} \cdot \jumpop{\nabla \mbu_h \otimes 
\mbn_e}
- \int_\Omega \mbR_h( \nabla \mbu_h ) \cdot (\mbph - \mbph_h)
 {=: I_1 - I_2}
 \end{aligned}
\end{equation}

The assumption that $|\mbu_h|_{H^2(\Omega, T_h)}$ is bounded for all $h >0$, implies 
that $\sum_{e \in E_h} h^{-1}_e \norm{\jumpop{\nabla \mbu_h}}^2_{L^2(e)^{2 \times 2}}$ 
is also uniformly bounded. Also, from Proposition \ref{prop:Rh_bound} the global
 lifting operator is bounded from the jump terms, inequality (\ref{eq:R_bound}).
As a result, 
\begin{align}
\int_\Omega |\mbR_h( \nabla \mbu_h) |^2 \le 
C_R \sum_{e \in E_h}h_e^{-1} \int_e |\jumpop{\nabla \mbu_h}|^2  \le C^\prime
\label{jump_term_bound}
\end{align}
and $(\mbR_h( \nabla \mbu_h))$ is uniformly bounded in the 
$L^2(\Omega)^{2 \times 2 \times 2}$ norm.
From the last relation and the Cauchy-Schwarz inequality we bound  $I_2$:
\begin{equation}
\begin{aligned}
|I_2| =& \Bigl| \int_\Omega \mbR_h( \nabla \mbu_h) \cdot (\mbph - \mbph_h) 
\Bigr|
\leq  
 { \norm{ \mbR_h( \nabla \mbu_h) }_{L^2(\Omega)^{2 \times 2 \times 2}} \norm{\mbph - 
\mbph_h }_{L^2(\Omega)^{2 \times 2 \times 2}} }.
\end{aligned}
\label{eq:R_I2}
\end{equation}
 Hence, $I_2 \rightarrow 0$, as $h \rightarrow 0$ by the error estimate given in (\ref{eq:phi_h_conv}). 
 Similarly for the term $I_1$, using the previous uniform bound for the jump terms, 
 the error estimate  for integrals that are defined over an edge $e$, (\ref{ineq:error_on_e}),
 and letting $K_e = \{ K\in T_h : e \in \partial K\}$, we infer that
 \begin{equation}
\begin{aligned}
|I_1| 
 &\le
\sum_{e \in E_h} \norm{ \dgal{\mbph -\mbph_h} }_{L^2(e)^{2 \times 2 \times 2}} 
\norm{\jumpop{\nabla \mbu_h}}_{L^2(e)^{2 \times 2}}  
\le
 c \sum_{e \in E_h } h_e^{1/2} |\mbph|_{H^1(K_e)^{2 \times 2 \times 2}} 
\norm{\jumpop{\nabla \mbu_h}}_{L^2(e)^{2 \times 2} } 
\\
 &\le
 {c \bigg(\sum_{e \in E_h } h_e^{2} |\mbph|^2_{H^1(K_e)^{2 \times 2 \times 2}}  \bigg)^{1/2}
 \bigg( 
 \sum_{e \in E_h } h_e^{-1}\norm{\jumpop{\nabla \mbu_h}}^2_{L^2(e)^{2 \times 2}} \bigg)^{1/2}
\le
C h |\mbph|_{H^1(\Omega)^{2 \times 2 \times 2}} \rightarrow 0,}
\end{aligned}
\label{eq:R_I1}
\end{equation}
Since $\mbu_h \rightarrow \mbu$ in $H^1(\Omega)$, taking the limit as $h \rightarrow 0$, 
employing   (\ref{eq:R_I2}) and  (\ref{eq:R_I1}),
the relation  (\ref{eq:weak_G_conv}) implies
\begin{equation}
\begin{aligned}
\lim_{h \rightarrow 0 }\int_\Omega G_h(\nabla \mbu_h) \cdot \mbph =&
\lim_{h \rightarrow 0 }\Big(-\int_\Omega \nabla \mbu_h \cdot (\nabla \cdot \mbph) 
+\mbR_h( \nabla \mbu_h ) \cdot \mbph 
\\ 
 &+  \sum_{e \in E_h}\int_{e} \mbph \cdot \jumpop{\nabla \mbu_h \otimes \mbn_e }
\Big)
 =  {-\int_\Omega \nabla \mbu \cdot (\nabla \cdot \mbph).}
\end{aligned}
\end{equation}

 \end{proof}

\begin{Corollary}{(Weak Convergence of the Discrete Gradients)}. 
\label{cor:discrete_gradient_weak_conv}
Suppose the assumptions of Lemma \ref{Lemma:Weak_D_gradient_weak_limit} hold. In addition assume that 
$\mbu \in H^2(\Omega)^2$, then 
\begin{align}
G_h(\nabla \mbu_h)  \rightharpoonup \nabla \nabla \mbu \quad \text{in } L^2(\Omega)^{2 \times 2 \times 2}.
\end{align}

\begin{proof}
Lemma \ref{Lemma:Weak_D_gradient_weak_limit} ensures the limit of equation (\ref{eq:discrete_gradient_distr_convergence}) for all $\mbph \in C^\infty_c(\Omega)$,  hence
\begin{equation}
\begin{aligned}
\lim_{h \rightarrow 0 } \bigg|  \int_\Omega \Big(G_h(\nabla \mbu_h)  - \nabla \nabla \mbu\Big) \cdot \mbph \bigg| 
\le  
\lim_{h \rightarrow 0 } \bigg|  \int_\Omega (\nabla \mbu_h  -  \nabla \mbu)\cdot (\nabla \cdot \mbph )\bigg|
= 0. \nonumber
\end{aligned}
\end{equation}

\end{proof}

\end{Corollary}

\begin{Theorem}[The $\liminf$ inequality.]
\label{Theorem:liminf}

Assume that $ {\alpha} > C_R$, i.e. that the  parameter of the discrete energy function
is larger than the constant of (\ref{eq:R_bound}). Also, let the penalty term
satisfy condition (\ref{pen:polynomial_growth}).
Then for all $\mbu \in \mathbb{A}(\Omega)$ and all sequences 
$(\mbu_h)\subset \mathbb{A}^q_h(\Omega)$
 such that $\mbu_h \rightarrow \mbu$ in $H^1(\Omega)$  it holds that
 \begin{align}
 \Psi[\mbu] \leq \liminf\limits_{h\rightarrow0} \Psi_h[\mbu_h]. 
\end{align}
\end{Theorem}
\begin{proof}

 We assume there is a subsequence, still denoted by $\mbu_h$, such 
 that $\Psi_h[\mbu_h] \leq C$ uniformly in $h$,
  otherwise $\Psi[\mbu] \leq \liminf\limits_{h\rightarrow0}
 \Psi_h[\mbu_h] = +\infty$. The following steps conclude the proof:
\begin{enumerate}[\labelsep=0.2em 1.]
\item From Proposition \ref{equi-coercivity}, the uniform bound $\Psi_h[\mbu_h] \le C$ implies 
that $|\mbu|_{H^2(\Omega, T_h)}$ and $\norm{\mbu_h}_{H^1(\Omega)}$ are uniformly bounded.

 \item Corollary   \ref{cor:discrete_gradient_weak_conv} implies 
 $G_h(\nabla \mbu_h)  \rightharpoonup \nabla \nabla \mbu$ in $L^2(\Omega)^{2 \times 2 \times 2}$.

 \item The term $\int_\Omega |G_h|^2$ is convex which
 implies weak lower semicontinuity \cite{dacorogna2007direct}: 
 $\liminf\limits_{h\rightarrow0}
\int_\Omega |G_h|^2 \ge \int_\Omega |\nabla \nabla \mbu|^2$.

\item From the Poincar\'e inequality for broken Sobolev spaces, (\ref{Poincare_ineq}),
 there exist a constant $c$ independent of $h$  such that for all $m \in [1, +\infty)$ it holds 
\begin{equation}
\begin{aligned}
\norm{\nabla \mbu_h}^2_{L^m(\Omega)^{2 \times 2 }} &\le 
c  \Big( | \mbu_h|_{H^2(\Omega, T_h)}^2 +
\Big|\frac{1}{|\Omega|} \int_\Omega \nabla \mbu_h \Big|^2 \Big) 
\\
& \le
C \Big( | \mbu_h|_{H^2(\Omega, T_h)}^2 +  
\norm{\nabla \mbu_h}_{L_2(\Omega)^{2 \times 2}}^2 \Big) 
 { <+\infty,}
\end{aligned}
\end{equation}
where the last bound holds from \textit{step 1}.

\item Finally, the assumed convergence of $\mbu_h$ to $\mbu$ in $H^1(\Omega)^2$,
 Lemma~\ref{Lemma:lo_convergence0} and Remark~\ref{Remark:penalty_convergence}
ensure the convergence of the remaining terms,
i.e. $\int_\Omega W(\nabla \mbu_h) + \Phi(\nabla \mbu_h) \rightarrow 
\int_\Omega W(\nabla \mbu) +  \Phi(\nabla \mbu)$.
\end{enumerate}
As a result we deduce that $\Psi[\mbu] \leq \liminf\limits_{h\rightarrow0} \Psi_h[\mbu_h]$.
\end{proof}

 \subsection{The $\limsup$ inequality}

In this section we focus on the $\limsup$ inequality. Given $\mbu \in \mathbb{A}(\Omega)$, we would like to prove the existence of  a sequence $(\mbu_h) \subset \mathbb{A}_h^q(\Omega)$, 
such that $\mbu_{h} \rightarrow \mbu$ in $H^1(\Omega)$ and 
$ \Psi[\mbu] \geq \limsup\limits \Psi_{h}[\mbu_{h}]$ as $h \rightarrow 0$. 
To this end, choose a sequence of smooth functions $\mbu_\delta$ such that
 $\mbu_\delta \rightarrow \mbu$ in $H^2(\Omega)$ and we define $\mbu_h$ to be 
 an appropriate interpolant of $\mbu_\delta$,  $I^q_h \mbu_\delta$, for a chosen $\delta$.
 A similar strategy was used in \cite{bartels2017bilayer}. We start with the following Proposition:

 \begin{prop}
\label{prop:jump_vanish} 
 For $\mbu \in H^s(\Omega)^2$, if $s > 2$ and $q \ge 2$, there exist  constants $C_1, C_2 >0$ 
 such that
 \begin{align}
&\sum_{e \in E_h^i} \frac{1}{h_e}\int_{e} |\jumpop{\nabla I_h^q \mbu} |^2
\le C_1 h^{2s-4}|\mbu|^2_{H^s(\Omega)} \quad \mbox{and}\label{eq:Intrp_jump1}
\\
 &\sum_{e \in E_h^i}\int_{e}\dgal{ \nabla \nabla I^q_h \mbu } \cdot 
 \jumpop{\nabla I^q_h \mbu\otimes \mbn_e} \le C_2 h^{s-2}|\mbu|_{H^2(\Omega)} |\mbu|_{H^s(\Omega)}. \label{eq:Intrp_jump2}
 \end{align}
 \end{prop}
\begin{proof} 
 To simplify the notation let $\mbv_h := I^q_h \mbu$. First we study  (\ref{eq:Intrp_jump1}). Adding and 
subtracting the term $\nabla \mbu$, using 
\tred{ the error estimates  of Corollary \ref{est_e_K}, yields}
 \begin{equation}
  \begin{aligned}
   \int_{e} |\jumpop{\nabla \mbv_{h} }|^2 
& \le 2 \int_{e} |\nabla \mbv_{h}^+ - \nabla \mbu|^2 + 2 \int_{e} |\nabla 
\mbv_{h}^- - \nabla \mbu|^2 
\leq Ch_e^{2s-3} |\mbu|^2_{H^s(K^+_e\bigcup K^-_e) },
\end{aligned}
\label{Intrp_jump1_prof1}
 \end{equation}
where $K^+_e,K^-_e$ denote the distinct elements that share the edge $e$, i.e.
$e = K^+_e \cap K^-_e$. Summing over all edges $e \in E^i_h$ 
inequality (\ref{Intrp_jump1_prof1}) gives  
 \begin{equation}
  \begin{aligned}
   \sum_{e \in E_h^i} h_e^{-1}\int_{e} |\jumpop{\nabla \mbv_h}|^2 
   &\le c
   \sum_{e \in E_h^i} h_e^{2s-4} |\mbu|^2_{H^s(K^+_e\bigcup K^-_e) } 
   \\
    &= c\sum_{e \in E_h^i} h_e^{2s-4} \Bigl(
 |\mbu|^2_{H^s(K^+_e) } + |\mbu|^2_{H^s(K^-_e) }\Bigr)
  {= C  h^{2s-4} |\mbu|^2_{H^s(\Omega) }.}
  \end{aligned}
 \end{equation}

To prove (\ref{eq:Intrp_jump2}) notice that
\begin{equation}
\begin{aligned}
&\int_{e} \dgal{ \nabla \nabla \mbv_h } \cdot \jumpop{\nabla \mbv_h \otimes \mbn_e}
= \frac{1}{2}\int_{e} \Bigl( \nabla \nabla \mbv^+_h
 + \nabla \nabla \mbv^-_h \Bigr) \cdot 
 \jumpop{\nabla \mbv_h \otimes \mbn_e }
\\
=& \frac{1}{2}\int_{e}  \nabla \nabla \mbv^+_h \cdot 
\jumpop{\nabla \mbv_h \otimes \mbn_e }
 + \nabla \nabla \mbv^-_h  \cdot 
 \jumpop{\nabla \mbv_h \otimes \mbn_e}. 
\end{aligned}
\end{equation}
The Cauchy Schwarz inequality and the discrete trace inequality (\ref{eq:trace_ineq}) imply
\begin{equation}
\begin{aligned}
&\int_{e}  \nabla \nabla \mbv^+_h \cdot \jumpop{\nabla \mbv_h \otimes \mbn_e }
\le \Big( \int_{e}  |\nabla \nabla \mbv^+_h|^2 \Big)^{1/2}
\Big( \int_{e}  |\jumpop{\nabla \mbv_h} |^2 \Big)^{1/2} 
\\ \le& 
\frac{C}{\sqrt{h_e}} \Big( \int_{K^+_e}  |\nabla \nabla \mbv_h|^2 \Big)^{1/2}
\Big( \int_{e}  |\jumpop{\nabla \mbv_h} |^2 \Big)^{1/2} 
 = 
 {C \Big( \int_{K^+_e}  |\nabla \nabla \mbv_h|^2 \Big)^{1/2}
\Big( \frac{1}{h_e} \int_{e}  |\jumpop{\nabla \mbv_h} |^2 \Big)^{1/2}.}
\end{aligned}
\label{div_term+}
\end{equation}
Therefore
\begin{equation}
\begin{aligned}
\Big| \sum_{e \in E_h^i}\int_{e}  \nabla \nabla \mbv^+_h \cdot \jumpop{\nabla \mbv_h \otimes \mbn_e }  \Big|
  \le&
   {c \sum_{e\in E_h^i} \Big( \int_{K^+_e}  |\nabla \nabla \mbv_h|^2 \Big)^{1/2}
\Big( \frac{1}{h_e} \int_{e}  |\jumpop{\nabla \mbv_h} |^2 \Big)^{1/2}}
\\ \le& 
c \Big(\sum_{e\in E_h^i}  \int_{K^+_e}  |\nabla \nabla \mbv_h|^2 \Big)^{1/2}
\Big( \sum_{e\in E_h^i} \frac{1}{h_e} \int_{e}  |\jumpop{\nabla \mbv_h} |^2 \Big)^{1/2}.
\end{aligned}
\label{jump2conv}
\end{equation}
However,
\begin{equation}
\begin{aligned}
 &\sum_{e\in E_h^i}  \int_{K^+_e}  |\nabla \nabla \mbv_h|^2 
  { \le 2 \sum_{e\in E_h^i}  \int_{K^+_e}  |\nabla \nabla \mbv_h - \nabla \nabla \mbu|^2
   +  2\sum_{e\in E_h^i}  \int_{K^+_e}|\nabla \nabla \mbu|^2 }
  \\ \le& 
2\sum_{e\in E_h^i} c |\mbu|^2_{H^2(K_e^+)}     + c_2 |\mbu|^2_{H^2(\Omega)}
   \le 
 { C  |\mbu|^2_{H^2(\Omega)}    + c_2 |\mbu|^2_{H^2(\Omega)}
 \le c_3 |\mbu|^2_{H^2(\Omega)} }
\end{aligned}
\end{equation}
where we have used the error estimates of (\ref{error_est}). Consequently (\ref{jump2conv})
becomes
\begin{equation}
\begin{aligned}
\Big| \sum_{e \in E_h^i}\int_{e}  \nabla \nabla \mbv^+_h \cdot \jumpop{\nabla \mbv_h\otimes \mbn_e}   \Big|
\le& 
 {
C |\mbu|_{H^2(\Omega)} 
\Big( \sum_{e\in E_h^i} \frac{1}{h_e} \int_{e}  |\jumpop{\nabla \mbv_h} |^2 \Big)^{1/2}  }
\\ \le&
c h^{s-2}|\mbu|_{H^2(\Omega)} |\mbu|_{H^s(\Omega)} .
\end{aligned}
\end{equation}

Similarly we show that
\begin{equation}
\sum_{e \in E_h^i} \int_{e}  \nabla \nabla \mbv^-_h \cdot \jumpop{\nabla \mbv_h \otimes \mbn_e } 
\le c h^{s-2}|\mbu|_{H^2(\Omega)} |\mbu|_{H^s(\Omega)}
\label{div_term-}
\end{equation}
\end{proof}

 \begin{Theorem}[The $\limsup$ inequality.]
\label{Theorem:limsup}
Let the penalty function $\Phi$ satisfy condition (\ref{pen:polynomial_growth}).
The following property holds:
 For all $\mbu \in\mathbb{A}(\Omega)$, there exists a sequence
   $(\mbu_h)_{h>0}$ with $\mbu_h \in \mathbb{A}^q_h(\Omega)$,  such that 
$\mbu_h \rightarrow \mbu$ in $H^1(\Omega)^2$ and 
\begin{align}
 \Psi[\mbu] \geq \limsup\limits_{h\rightarrow 0} \Psi_h[\mbu_h]. 
\end{align}
\end{Theorem} 
\begin{proof}
 Given $\mbu \in \mathbb{A}(\Omega)$ we construct an appropriate sequence \tred{$(\mbu_h)_{h>0}\subset \mathbb{A}^q_h(\Omega)$,} the recovery sequence, such that
 $\norm{\mbu_h - \mbu}_{H^1(\Omega)^2} \rightarrow 0$ and $ \lim_{h\rightarrow 0} \Psi_h[\mbu_h] = \Psi[\mbu] $.
 
 In particular, we approximate $\mbu$ via mollification by a sequence of smooth functions in $\overline{\Omega}$
 to find that, for all $\delta > 0$, 
there exists \tred{$\mbu _\delta\in H^3(\Omega)^2\cap \mathbb{A}(\Omega)$} such that 
\begin{align}
\norm{\mbu_\delta - \mbu}_{H^2(\Omega)} < c\delta 
\quad \text{and} \quad
|\mbu_\delta|_{H^3(\Omega)} < \frac{c}{\delta} |\mbu|_{H^2(\Omega)}.
\label{ud_conv}
\end{align}
Recall, $\mbu_\delta =  \mbu = \mbg$  and  $\nabla \mbu_\delta = \nabla \mbu = \nabla \mbg, $ on $\partial\Omega$, where 
$\mbg \in H^3(\Omega)$ by the definition of $\mathbb{A}(\Omega)$.
Here $c$ is independent of $\delta$.

Next, define $\mbu_{h, \delta} := I^q_h \mbu_\delta \in H^1(\Omega)^2$,
noting that $\mbu_{h, \delta} \in \mathbb{A}^q_h(\Omega)$.
From the error estimates in (\ref{error_est}) and the fact that $q \ge 2$ we find that
\begin{align}
\norm{\mbu_{h, \delta}  -\mbu_\delta }_{H^1(\Omega)} &\leq c h |\mbu_\delta |_{H^2(\Omega)}
\label{eq:uhd_error0}
\\
|\mbu_{h, \delta}  -\mbu_\delta |_{H^2(K)} &\leq C |\mbu_\delta |_{H^2(K)}, \quad  \text{ } \text{ }K \in T_h,
\label{eq:uhd_error1}
\\
|\mbu_{h, \delta}  -\mbu_\delta |_{H^2(K)} &\leq C h |\mbu_\delta |_{H^3(K)}, \quad K \in T_h.
\label{eq:uhd_error}
\end{align}

Inequality (\ref{ud_conv}) and the error estimate of (\ref{eq:uhd_error0}) imply that
\begin{equation}
\begin{aligned}
 &\norm{\mbu_{h, \delta}  -\mbu }_{H^1(\Omega)} 
 \leq \norm{\mbu_{h, \delta}  -\mbu_\delta }_{H^1(\Omega)} 
+ \norm{ \mbu_\delta  -\mbu }_{H^1(\Omega) } 
\leq  h|\mbu_\delta |_{H^2(\Omega)} + \norm{ \mbu_\delta  -\mbu }_{H^2(\Omega) }  \\
& 
 {\leq   h(  |\mbu_\delta  -\mbu|_{H^2(\Omega)} + |\mbu |_{H^2(\Omega)} )+ 
\norm{ \mbu_\delta  -\mbu }_{H^2(\Omega) } } 
\leq  h(  \delta + |\mbu |_{H^2(\Omega)} )+ 
\delta .
\end{aligned}
\label{diag_arg}
\end{equation}
Choosing 
\begin{align}
\delta = \sqrt{h},
\label{delta2h}
\end{align}
we deduce that the sequence
 $(\mbu_{h, \delta_h})\subset V^q_h(\Omega) $
and $\mbu_{h, \delta_h} \rightarrow \mbu$ in $H^1(\Omega)$,
as $h \rightarrow 0$.
It remains to prove
\begin{align}
| \Psi_h[\mbu_{h, \delta_h}] - \Psi[\mbu] | \rightarrow 0, \quad \text{as } h\rightarrow 0.
\label{eq:en_limsup}
\end{align}

For the convergence of the lower order terms we use Lemma \ref{Lemma:lo_convergence0} and Remark \ref{Remark:penalty_convergence}.  Hence, it suffices to show that   
$\norm{\nabla \mbu_{h, \delta_h}}_{L^q(\Omega)^{2 \times 2}}$,
is uniformly bounded with respect to $h$. Indeed, 
Sobolev's embedding theorem
and classical error 
estimates, see inequality (\ref{eq:uhd_error1}), imply that  
 $\mbu_{h, \delta_h}$ is uniformly bounded in ${W^{1,q}(\Omega)} $ with respect
 to $h$, for all $q\in [1, \infty)$:
 \begin{equation}
\begin{aligned}
 &\norm{ \mbu_{h, \delta_h}}_{W^{1,q}(K)} 
 \leq \norm{\mbu_{\delta_h} - \mbu_{h, \delta_h} }_{W^{1,q}(K)} + \norm{\mbu_{\delta_h}}_{W^{1,q}(K)}
 \\
& \leq 
c_1 \norm{\mbu_{\delta_h} - \mbu_{h, \delta_h}}_{H^2(K)} +  c_2\norm{\mbu_{\delta_h}}_{H^2(K)} 
 { \leq 
   c \norm{\mbu_{\delta_h} }_{H^2(K)}. }
 \end{aligned}
 \end{equation}
To extend the bound over $\Omega$, assume that $q$ is integer
and $q \ge 2$; then the multinomial formula implies 
 \begin{equation}
\begin{aligned}
& \norm{ \mbu_{h, \delta_h}}_{W^{1,q}(\Omega)} = 
 \Big( \sum_{K \in T_{h} }\norm{ \mbu_{h, \delta_h} }_{W^{1,q}(K)} ^q \Big)^{1/q}
 {\le \Big(  c^q  \sum_{K \in T_{h} } \norm{\mbu_{\delta_h} }_{H^2(K)}^q \Big)^{1/q} }
 \\
 &\le c  \Big( \big(\sum_{K \in T_{h} } \norm{\mbu_{\delta_h}}_{H^2(K)}^2 \big)^{q/2} 
 \Big)^{1/q}
  { =  c   \norm{\mbu_{\delta_h}}_{H^2(\Omega)} }
 \\
 & \le c \big( \norm{\mbu_{\delta_h} - \mbu}_{H^2(\Omega)} + \norm{\mbu}_{H^2(\Omega)} \big)
 { \le c \big(\delta_h + \norm{\mbu}_{H^2(\Omega)} \big). }
 \end{aligned}
 \label{in:uhd_unbound}
 \end{equation}
The result can be extended for $q=1$ because $|\Omega| < +\infty$ and for $q \in [1, +\infty)$
using the interpolation bound, \cite{brezis2010functional}:
\[
 \norm{\mbu_h}_{L^q(\Omega) } \le \norm{\mbu_h}_{L^{\lfloor q \rfloor}(\Omega)}^\xi
 \norm{\mbu_h}_{L^{\lceil q \rceil }(\Omega)}^{1-\xi},
\]
where $\lfloor \cdot \rfloor, \lceil \cdot \rceil$ are the floor and ceiling integer functions respectively and 
$\xi = \frac{(\lceil q \rceil - q) \lfloor q \rfloor}{q}$.
 Notice that Lemma \ref{Lemma:lo_convergence0} and Remark \ref{Remark:penalty_convergence} give 
 \begin{align}
  W(\nabla \mbu_{h, \delta_h}) \rightarrow W(\nabla \mbu)
\text{ and } 
\Phi(\nabla \mbu_{h, \delta_h} ) \rightarrow \Phi(\nabla \mbu), 
\label{limsup:Phi_convergence}
 \end{align}
in $L^1(\Omega)$, as $h \rightarrow 0$. 
Also, since we have choosen $\delta_h = \sqrt{h}$, 
Proposition \ref{prop:jump_vanish},  ($\ref{ud_conv}$),
 implies that as  $h \rightarrow 0$ 
 \begin{equation}
 \begin{aligned}
\sum_{e \in E_{h}^i} \frac{1}{h_e}\int_{e} |\jumpop{\nabla  \mbu_{h, \delta_h}} |^2
\le C_1 h^{2}|\mbu_{\delta_h}|^2_{H^3(\Omega)} 
 {\le c \frac{h^2}{\delta_h^2}|\mbu|^2_{H^2(\Omega)}  \le c h|\mbu|^2_{H^2(\Omega)} \rightarrow 0.}
\label{limsup:jump_conv}
 \end{aligned}
 \end{equation}
 \tred{For the boundary terms we have:
 \begin{equation}
 \begin{aligned}
\sum_{e \in E_{h}^b} &\frac{1}{h_e}\int_{e} |\jumpop{\nabla  \mbu_{h, \delta_h}} |^2= \sum_{e \in E_{h}^b} \frac{1}{h_e}\int_{e} | \nabla  \mbu_{h, \delta_h} - \nabla I_h^q   \mbg |^2\\
&\leq 2 \sum_{e \in E_{h}^b} \frac{1}{h_e}\int_{e} \Big ( | \nabla  \mbu_{h, \delta_h} - \nabla  \mbu_{\delta_h}  |^2 + | \nabla  \mbg - \nabla I_h^q   \mbg |^2 \Big )
\le C_1 h^{2}( |\mbu_{\delta_h}|^2_{H^3(\Omega)}  +  |\mbg|^2_{H^3(\Omega)} )\\
& {\le c \frac{h^2}{\delta_h^2}|\mbu|^2_{H^2(\Omega)}  + c h^{2}   |\mbg|^2_{H^3(\Omega)}   \rightarrow 0.}
\label{limsup:jump_conv_2}
 \end{aligned}
 \end{equation}}
 Similarly, 
 \begin{equation}
 \begin{aligned}
\sum_{e \in E_{h}}\int_{e}\dgal{ \nabla \nabla  \mbu_{h, \delta_h} } \cdot 
 \jumpop{\nabla  \mbu_{h, \delta_h} \otimes \mbn_e} \le
  C_2 h|\mbu_{\delta_h}|_{H^2(\Omega)}( |\mbu_{\delta_h}|_{H^3(\Omega)}  +  |\mbg|_{H^3(\Omega)} )
  \\
  \le  C h^{1/2}\left( |\mbu|_{H^2(\Omega)} + c\delta_h \right)( |\mbu_{\delta_h}|_{H^2(\Omega)}  +  |\mbg|_{H^3(\Omega)} )\rightarrow 0.
 \end{aligned}
 \end{equation}

 For the remaining higher order terms, we work similarly as before. We begin, showing an inequality
 for a given element $K$ using the error estimates of (\ref{eq:uhd_error1}) and (\ref{eq:uhd_error}):
\begin{equation}
 \begin{aligned}
   \big| |\mbu_{\delta_h}|^2_{H^2(K)} -|\mbu_{h, \delta_h}|^2_{H^2(K)} \big| 
  =&
  \big| |\mbu_{\delta_h}|_{H^2(K)} -|\mbu_{h, \delta_h}|_{H^2(K)} \big|
  \big(|\mbu_{\delta_h}|_{H^2(K)} +|\mbu_{h, \delta_h}|_{H^2(K)} \big) 
  \\
  \le&
  \big| |\mbu_{\delta_h}|_{H^2(K)} -|\mbu_{h, \delta_h}|_{H^2(K)} \big|
  \big(2|\mbu_{\delta_h}|_{H^2(K)} +|\mbu_{h, \delta_h} -\mbu_{\delta_h} |_{H^2(K)} \big) 
  \\
  \le&
  |\mbu_{\delta_h} -\mbu_{h, \delta_h}|_{H^2(K)}
  \big(2|\mbu_{\delta_h}|_{H^2(K)} + c_2|\mbu_{\delta_h}|_{H^2(K)} \big)
  \\
  \le&
  C_2 |\mbu_{\delta_h} -\mbu_{h, \delta_h}|_{H^2(K)}
  |\mbu_{\delta_h}|_{H^2(K)}
\le
 {  ch |\mbu_{\delta_h}|_{H^3(K)} |\mbu_{\delta_h}|_{H^2(K)}.  }
  \end{aligned}
  \nonumber
\end{equation}
Then summing over $K$ in $T_h$ and using the Cauchy-Schwarz inequality, gives:
\begin{equation}
 \begin{aligned}
  &\big| \sum_{K\in T_{h}}|\mbu_{\delta_h}|^2_{H^2(K)} -\sum_{K\in T_h}|\mbu_{h, \delta_h}|^2_{H^2(K)} \big| 
  \le 
  { \sum_{K\in T_{h}} \big| |\mbu_{\delta_h}|^2_{H^2(K)} -|\mbu_{h, \delta_h}|^2_{H^2(K)} \big| }
  \\
  \le&
\sum_{K\in T_{h}} \Big( 
  ch|\mbu_{\delta_h}|_{H^3(K)} |\mbu_{\delta_h}|_{H^2(K)} 
  \Big)
\le ch|\mbu_{\delta_h}|_{H^3(\Omega)} |\mbu_{\delta_h}|_{H^2(\Omega)} 
 {\le c\sqrt{h}|\mbu|_{H^2(\Omega)} |\mbu_{\delta_h}|_{H^2(\Omega)} \rightarrow 0,}
 \end{aligned}
\label{uhd_high_order}
\end{equation}
$\text{as } h \rightarrow 0$, where for the last inequality we have used ($\ref{ud_conv}$)
 and (\ref{delta2h}).
Furthermore, note that
\begin{equation}
 \begin{aligned}
  &\big| |\mbu_{\delta_h}|^2_{H^2(\Omega)} -|\mbu|^2_{H^2(\Omega)} \big| 
  \le |\mbu_{\delta_h} -\mbu|_{H^2(\Omega)}
  \big( |\mbu_{\delta_h}|_{H^2(\Omega)} +|\mbu|_{H^2(\Omega)} \big)
 \\
  &\le |\mbu_{\delta_h} -\mbu|_{H^2(\Omega)}
  \big( |\mbu_{\delta_h} -\mbu|_{H^2(\Omega)} +2|\mbu|_{H^2(\Omega)} \big)
 {  \le \delta_h  \big( \delta_h +2|\mbu|_{H^2(\Omega)} 
\big)
\le 2 \delta_h |\mbu|_{H^2(\Omega)} + \delta_h^2. }
  \end{aligned}
  \label{ud_high_order}
\end{equation}
Finally (\ref{uhd_high_order}) and  (\ref{ud_high_order}) give 
\begin{align}
\left|  \sum_{K \in T_{h}}|\mbu_{h, \delta_h}|^2_{H^2(K)} -|\mbu|^2_{H^2(\Omega)} \right| \rightarrow 0, 
\quad \text{as } h\rightarrow 0,
\end{align}
which concludes (\ref{eq:en_limsup}).
\end{proof}

\section{Compactness and Convergence of Discrete Minimizers}
\label{sec:compactness_convergence}

In this section our main task is to use the results of the previous section to show that under some
boundedness hypotheses on $\mbu_h$, a sequence of discrete minimizers $(\mbu_h)$
converges in $H^1(\Omega)$ to a global minimizer $\mbu$ of the continuous functional,
Theorem~\ref{Thm:min_convergence}. Such results are standard 
in the $\Gamma-$convergence literature, \cite{braides2002gamma,dal2012introduction},
 but the application in our setting is not straightforward. The main reason is that
in our case we need certain intermediate results, such as a discrete DG version of the 
Rellich-Kondrachov theorem, which we show in the sequel. We use related discrete
bounds derived previously in 
\cite{lew2004optimal,eymard2009discretization, di2010discrete, buffa2009compact}

We will use the \textit{total variation} of a function $v \in L^1(\Omega)$,
see \cite{evans2015measure}, defined as
\begin{align}
 |Dv|(\Omega) = \sup \left\{ \int_\Omega v (\nabla \cdot \mbphi) \,:\, \mbphi \in C_c^1(\Omega)^2,
\norm{ \mbphi}_{L^\infty(\Omega)} \le1\right\}.
\end{align}

The space of functions of \textit{bounded variation} in $\Omega$, denoted by $BV(\Omega)$, contains
all  $L^1(\Omega)$ functions with bounded total variation, i.e.,
\begin{align}
BV(\Omega) = \{ v \in L^1(\Omega)\,:\,| D v|(\Omega) < +\infty \}.
\label{def:BV}
\end{align}
The space $BV(\Omega)$ is endowed with the norm 
$\norm{v}_{BV} = \norm{v}_{L^1(\Omega)} +  |Dv|(\Omega)$.
The following inequality plays a key role in the desired compactness.

\begin{Lemma}[Bounds for the total variation]
\label{Lemma:total_variayion_bound}
 Let $\mbw \in V^q_h(\Omega)^2$. Then, there exist a constant $C$ independent of $h$ 
 such that
 \begin{align}
 |D \nabla \mbw|(\Omega) \le C |\mbw|_{H^2(\Omega, T_h)}.
\label{eq:total_variation_bound}
 \end{align}
A proof can be found in \cite[Theorem 3.26]{lew2004optimal} and a 
generalization in (\cite[Lemma 2]{buffa2009compact}.
It is based on the observation
\begin{equation}
\int_\Omega w_{i,j} \phi_{ijk, k} = \sum_{e \in E^i_h}\int_e \phi_{ijk} 
\jumpop{\nabla  \mbw \otimes \mbn_e}_{ijk} 
- \sum_{K \in T_h} \int_K w_{i,jk} \phi_{ijk},
\end{equation}
where $\mbphi \in C_c^1(\Omega)^{2 \times 2 \times 2}$ 
and appropriate bounds on the right-hand side.
\end{Lemma}


 \begin{prop}[Discrete Rellich-Kondrachov]
 \label{prop:discr_Rellich-Kondrachov}
Let a sequence \tred{$(\mbu_h) \subset  \mathbb{A}^q_h(\Omega)$} be bounded, for $C>0$, as
\begin{align}
 \norm{\mbu_h}_{H^1(\Omega) } + |\mbu_h |_{H^2(\Omega, T_h)}  < C, \quad \text{for all } h > 0.
 \label{eq:norms_bound_compacntess}
\end{align}
Then $(\mbu_h)$ is relatively compact in $W^{1,p}(\Omega)^2$ for $1 \le p < +\infty$,
i.e. there exists a $\mbu \in W^{1,p}(\Omega)^2$ such that
\begin{align}
\mbu_h \rightarrow \mbu \quad \text{in } W^{1,p}(\Omega)^2,
\end{align}
up to a subsequence. 
 \end{prop}
\begin{proof}
Using the Sobolev embedding theorem, the discrete Poincar\'e inequality 
(\ref{Poincare_ineq})  and inequality (\ref{eq:norms_bound_compacntess}) 
we conclude that
\begin{equation}
\begin{aligned}
\norm{\mbu_h}_{L^r(\Omega)^2} &\le C \norm{\mbu_h}_{H^1(\Omega)^2} < C
\\
\norm{\nabla \mbu_h}_{L^r(\Omega)^{2 \times2}}^2 &\le c \big(   |\mbu_h 
|_{H^2(\Omega, T_h)}^2
+ \norm{ \nabla \mbu_h }_{L^2(\Omega)^{2 \times 2  } } ^2 \big) < C,
\end{aligned}
\label{eq:Wp_bounds}
\end{equation}
uniformly with respect to $h$, for all $r \in [1, +\infty)$, where $C$ is a positive constant independent of $h$.
 The space $W^{1,r}(\Omega)$ is reflexive for $r \in (1, +\infty)$. Therefore, for every 
bounded sequence there exists a subsequence, not relabeled, and a function $\mbu 
\in W^{1,r}(\Omega)$ such that 
\begin{align}
 \mbu_h \rightharpoonup \mbu \text{ in  } W^{1,r}(\Omega), 
 \text{for all }  r\in (1, +\infty).
 \label{eq:uh_weak_conv}
\end{align}
From the Rellich-Kondrachov  theorem, \cite[Theorem 9.16]{brezis2010functional}, and since $\dim{\Omega} =2$,
it is known that  $ {H^1(\Omega)  \subset \subset L^p(\Omega)} $, i.e. it is 
compactly embedded for $1 \le  p < +\infty$. In particular, $\mbu_h \rightarrow \mbu $ in $L^p(\Omega)$ for $p \in [1, +\infty)$.

It remains to prove that $(\nabla \mbu_h)$ is relatively compact in $L^p(\Omega)$.
Similar results have been proved in  
\cite{di2010discrete, buffa2009compact,eymard2009discretization},  
all of them are based  on a bound of the 
$BV$ norm, $\norm{\cdot}_{BV(\Omega)}$, from the higher order terms. To this end,
 from inequalities  (\ref{eq:total_variation_bound}) and (\ref{eq:norms_bound_compacntess})   
 it follows that $(\nabla \mbu_h)$ is uniformly bounded in $BV(\Omega)^{2\times 2}$. 
By standard embedding theorems, see for example \cite[Theorem 5.5]{evans2015measure}, $\left(\nabla \mbu_h \right)$ 
in $BV(\Omega)^{2 \times 2}$ is relatively compact in $L^1(\Omega)^{2 \times 2}$ and, up to 
a subsequence, there exists
$\mbw \in BV(\Omega)^{2 \times 2} $ such that 
\begin{align}
\nabla \mbu_h \rightarrow  \mbw,  \quad \text{in } L^1(\Omega)^{2 \times 2}.
\end{align}

Therefore, from the interpolation inequality, \cite{brezis2010functional},
 and (\ref{eq:Wp_bounds}), we obtain that
\begin{equation}
\begin{aligned}
\norm{\mbw - \nabla \mbu_h}_{L^p(\Omega)^{2 \times 2}}
& \le 
\norm{\mbw - \nabla \mbu_h}_{L^1(\Omega)^{2 \times 2}}^\theta
 \norm{\mbw - \nabla \mbu_h}_{L^r(\Omega)^{2 \times 2}}^{1 - \theta} 
\le 
 {C \norm{\mbw - \nabla \mbu_h}_{L^1(\Omega)^{2 \times 2}}^\theta \rightarrow 0, }
\end{aligned}
\end{equation}
where $p \in (1, r)$ and $\theta = \frac{r-p}{p(r-1)} \in (0,1)$.
From (\ref{eq:uh_weak_conv}) we conclude that $\mbu_h \rightarrow 
\mbu$ in $ W^{1,p}(\Omega)$ for $p \in [1,+\infty)$.
\end{proof}

The discrete Rellich-Kondrachov Theorem ensures that  there  exist $\mbu$ in $H^1(\Omega)^{2\times 2}$
such that $\mbu_h \rightarrow \mbu$ in $H^1(\Omega)$ under some boundedness 
hypotheses on $\mbu_h$. However, the proof that minimizers of the discrete problem
converge to a minimizer of the continuous problem would require higher
regularity on $\mbu$, 
i.e. $\mbu \in H^2(\Omega)^2$. Similar arguments were used previously in 
\cite{di2010discrete, buffa2009compact}.

  \begin{prop}[Regularity of the limit and Weak Convergence of the Discrete 
Gradient]
  \label{Proposition:Weak_D_gradient_conv}
  Let a sequence \tred{$(\mbu_h) \subset  \mathbb{A}^q_h(\Omega)$}
  with $\mbu_h \rightarrow \mbu$ in $H^1(\Omega)^2$.
If the sequence is bounded in the $H^2(\Omega, T_h)$ seminorm, then  
  \begin{equation}
  \begin{aligned}
  \mbu \in H^2(\Omega)^2
  \quad \text{ and } \quad
   {G_h(\nabla\mbu_h)  \rightharpoonup \nabla \nabla \mbu \text{ in }
  L^2(\Omega)^{2 \times 2 \times 2},}
  \end{aligned}
  \end{equation}
up to a subsequence.
In addition, 
$\mbu \in \mathbb{A}(\Omega)$.
%
\end{prop}
\begin{proof}

Here we adopt partially the proof of \cite{di2010discrete}. From inequality (\ref{eq:discrete_gradient_bound}), the discrete gradient $G_h(\nabla\mbu_h)$ is bounded by $|\mbu_h|_{H^2(\Omega,T_h)}$
  in the $L^2(\Omega)^{2 \times 2 \times 2}$ norm. Hence, there exists $\mbw \in L^2(\Omega)^{2 \times 2 \times 2}$ such that, up to a subsequence,
\begin{equation}
G_h(\nabla\mbu_h)  \rightharpoonup  \mbw  \text{ in } 
L^2(\Omega)^{2 \times 2 \times 2}.
\end{equation}
To prove that $\mbw = \nabla \nabla \mbu$, let $\mbph  \in 
C^\infty_c(\Omega)^{2 \times 2 \times 2}$. By Lemma \ref{Lemma:Weak_D_gradient_weak_limit}
 
\begin{equation}
\begin{aligned}
 \int_\Omega \mbw \cdot \mbph &= \lim_{h \rightarrow 0} \int_\Omega G_h(\nabla \mbu_h) \cdot \mbph  
 =- \int_\Omega \nabla \mbu \cdot  (\nabla \cdot \mbph),
\end{aligned}
\end{equation}
which means that $\mbw = \nabla \nabla \mbu$.

It remains to prove that $\mbu \in \mathbb{A}(\Omega)$. If $\mbu_h\in \mathbb{A}^q_h(\Omega)$
then $\restr{\mbu_h}{\partial \Omega} = \restr{(I^q_h \mbg)}{\partial \Omega}$ and $\mbg \in  {H^3(\Omega)^2}$.
Classical interpolation error estimates, 
see inequality (\ref{ineq:error_on_e}), give
\begin{equation}
\begin{aligned}
\norm{\mbu_h -  \mbg}_{L^2(\partial \Omega)} ^2 &= 
\norm{I^q_h \mbg -  \mbg}_{L^2(\partial \Omega)} ^2 =
\sum_{e\in E^b_h} \norm{I^q_h\mbg -  \mbg}^2_{L^2(e)}  
\\
&
 {\le c \sum_{e\in E^b_h} h_K^3  |\mbg|^2_{H^2(K)}    
\le C h^3 |\mbg|^2_{H^2(\Omega)} \rightarrow 0. }
\end{aligned}
\end{equation}
Combining the last result with the trace inequality yields 
\begin{equation}
\begin{aligned}
\norm{\mbu -  \mbg}_{L^2(\partial \Omega)}  & \le
\norm{\mbu -  \mbu_h}_{L^2(\partial \Omega)}  + \norm{\mbu_h -  \mbg}_{L^2(\partial \Omega)} 
\\
&\le c \norm{\mbu -  \mbu_h}_{H^1(\Omega)}  + \norm{\mbu_h -  \mbg}_{L^2(\partial \Omega)}  
\rightarrow 0, 
\end{aligned}
\end{equation}
which means $\mbu = \mbg $ a.e on $\partial \Omega$.
\tred{For the $\nabla \mbg$ term we first notice that  
\begin{equation}
\begin{aligned}
\norm{\nabla \mbu_h - \nabla \mbg}_{L^2(\partial \Omega)} ^2 &\leq  
2 \norm{\nabla I^q_h \mbg -\nabla   \mbg}_{L^2(\partial \Omega)} ^2 + 2 \norm{\nabla \mbu_h - \nabla I^q_h \mbg }_{L^2(\partial \Omega)} ^2\\
&=
2\sum_{e\in E^b_h} \norm{\nabla I^q_h \mbg -\nabla   \mbg}^2_{L^2(e)}  +\norm{\nabla \mbu_h - \nabla I^q_h \mbg }_{L^2(e)} ^2
\\
&
\le C h^3 |\mbg|^2_{H^3(\Omega)}  + C h |\mbu_h|_{H^2(\Omega,T_h)} \rightarrow 0. 
\end{aligned}
\end{equation}
On the other hand, Lemma \ref{Theorem:Trace_inequality} implies ($K_e$ is the element with boundary face $e$)
\begin{equation}
\begin{aligned}
\norm{\nabla \mbu_h - \nabla \mbu}_{L^2(\partial \Omega)} ^2 
&=
\sum_{e\in E^b_h} \norm{\nabla \mbu_h - \nabla \mbu}^2_{L^2(e)}  \leq C \sum_{e\in E^b_h} \norm{\nabla \mbu_h - \nabla \mbu}_{L^2(K_e)}  \norm{\nabla \mbu_h - \nabla \mbu}_{H^1(K_e)}
\\
&
\le C  \|  \mbu_h -   \mbu \|_{H^1(\Omega)}  (  \|  \mbu \|_{H^2(\Omega)} +  |\nabla \mbu_h  |_{H^1(\Omega)} +  |\mbu_h|_{H^2(\Omega,T_h)} ) \rightarrow 0. 
\end{aligned}
\end{equation}
The proof is thus complete. }
 \end{proof}

\begin{Theorem}[Convergence of discrete  absolute minimizers]
\label{Thm:min_convergence}
Assume that $ {\alpha} > C_R$, i.e. that the stabilization parameter is greater than the constant of (\ref{eq:R_bound}).
Let $(\mbu_h) \subset \mathbb{A}^q_h(\Omega)$ be a sequence
of  absolute minimizers of $\Psi_h$, i.e.,
\begin{align}
\Psi_h[\mbu_h] = \inf_{\mbw_h \in \mathbb{A}^q_h(\Omega)} \Psi_h[\mbw_h]. 
\label{eq:discrete_min}
\end{align}
If $\Psi_h[\mbu_h]$ is uniformly bounded then, up to a subsequence, there exists $\mbu \in \mathbb{A}(\Omega)$ such that
\begin{align}
\mbu_{h} \rightarrow \mbu,  \text{ in } H^1(\Omega)^2,
\text{ and } 
\Psi[\mbu] = \min_{\mbw \in \mathbb{A}(\Omega)} \Psi[\mbw].
\end{align}
\end{Theorem}

\begin{Remark}
We note that for each fixed $h>0$ the functional $\Psi_h$ admits a minimizer. Indeed, given an infimising sequence in the finite-dimensional space $\mathbb{A}^q_h(\Omega)$, by coercivity (Proposition \ref{equi-coercivity}) and the Poincar\'e inequality \eqref{Poincare_ineq}, it must be bounded in the norm
\[
\|\cdot\|^2 : = \|\cdot\|^2_{H^1(\Omega)} + |\cdot|^2_{H^2(\Omega,T_h)}.
\]
Finite-dimensionality of $\mathbb{A}^q_h(\Omega)$ implies strong convergence of the infimising sequence in the above norm and it is easy to check that the functional $\Psi_h$ is continuous with respect to this convergence.
\end{Remark}

\begin{proof}[Proof of Theorem \ref{Thm:min_convergence}]
The uniform bound for the discrete energies implies from the equi-coercivity property, 
Proposition \ref{equi-coercivity}, that
\begin{align}
 \norm{\mbu_h}_{H^1(\Omega)^2} + |\mbu_h|_{H^2(\Omega, T_h)} < C,  
\end{align} 
uniformly with respect to $h$.  The discrete Rellich-Kondrachov, Proposition 
\ref{prop:discr_Rellich-Kondrachov}, ensures that there exists $\mbu\in H^1(\Omega)^2$ 
such that $\mbu_h \rightarrow \mbu$ in 
$H^1(\Omega)^2$ up to a subsequence not relabeled here.
Since $|\mbu_h|_{H^2(\Omega, T_h)}$ is uniformly bounded,
Proposition  \ref{Proposition:Weak_D_gradient_conv} implies that 
$\mbu \in H^2(\Omega)^2$ and also $\mbu \in \mathbb{A}(\Omega)$.

To prove that $\mbu$ is a global minimizer of $\Psi$ we use the $\liminf$ and $\limsup$
inequalities, Theorems \ref{Theorem:liminf} and \ref{Theorem:limsup} respectively. 
Let $\mbw \in \mathbb{A}(\Omega)$, then the $\limsup$ inequality implies that  there exist
$\mbw_h \in \mathbb{A}^q_h(\Omega)$ such that 
\begin{align}
\mbw_h \rightarrow \mbw \text{ in } H^1(\Omega)^2 \quad \text{and} \quad 
 \limsup_{h \rightarrow 0} \Psi_h[\mbw_h] \le \Psi[ \mbw ].
\end{align}
Therefore, since $\mbu_h \rightarrow \mbu$ in $H^1(\Omega)$  the $\liminf$ inequality 
and the fact that $\mbu_h$ are absolute minimizers of the discrete problems imply that
\begin{align}
\Psi[ \mbu ] \le  \liminf_{h \rightarrow 0} \Psi_h[\mbu_h]  \le  \limsup_{h \rightarrow 0} \Psi_h[\mbu_h] 
\le  \limsup_{h \rightarrow 0} \Psi_h[\mbw_h] \le \Psi[ \mbw ], 
\end{align}
for all $\mbw \in \mathbb{A}(\Omega)$. Therefore $\mbu$ is an absolute minimizer of $\Psi$.
\end{proof}

\section{Incorporating Penalty terms with Exponential Growth}
\label{sec:Orlicz_embedding}

So far we have shown the convergence of discrete absolute minimizers to a global minimizer
of the continuous problem. The proof is based on the $\Gamma-$convergence of  
$\Psi_h$ to $\Psi$ and on discrete compactness results.
To penalize interpenetration of matter we have added a penalty function $\Phi$ in the 
total potential energy (\ref{equ:total_potential}), assuming that $\Phi$ 
has polynomial growth, see (\ref{pen:polynomial_growth}). 
One can notice that the polynomial growth penalizes also deformations where $J > 1$, 
 {$J = \det(\bm1 + \nabla \mbu)$ is the Jacobian determinant of the mapping}, 
and as a consequence can affect the material properties.

Using a penalty of the form 
\begin{align}
\Phi(\nabla \mbu) = \bar{\Phi}(J) = e^{ {A} \left(  {b}- J\right)}, \quad   0<  b <1, 
\label{equ:pen_assumption}
\end{align}
for large enough $ {A}>0$,  the penalty parameter contributes to the total potential energy when
$J <1$, thus an assumption as (\ref{equ:pen_assumption}) seems preferable.
 In addition, computational results related especially to densified phase 
 and its comparison to experimental data indicates that (\ref{equ:pen_assumption})
 is a better choice, for more details see \cite{grekasPhD}.

In this section we will assume that instead of  polynomial growth,
(\ref{equ:pen_assumption}) holds. However, employing a penalty function $\Phi$ with
 exponential growth, affects the proofs of the inequalities 
$\liminf$, Theorem \ref{Theorem:liminf}, and  $\limsup$, Theorem \ref{Theorem:limsup}.
The main technical difficulty addressed in this section is the proof of the $\limsup$ inequality
when (\ref{equ:pen_assumption}) is assumed. To show the analog of Theorem~\ref{Theorem:limsup},
and in particular that $\Phi$ is uniformly integrable one has to use appropriate 
Orlicz spaces and corresponding embedding results. To do this in
discrete DG spaces requires new ideas, which we describe in this section. 
As far as we know these bounds are the first embedding estimates for DG spaces
using the Orlicz framework.

Below we recall some definitions and basic properties for Orlicz spaces which we require. The 
reader is referred to \cite{orlicz_book} for a thorough review of Orlicz spaces.
Let $\varphi:\R\to[0,+\infty]$ be a continuous, convex and even function satisfying
\[
\lim_{t\to0}\frac{\varphi(t)}{t} = 0\mbox{ and }\lim_{t\to\infty}\frac{\varphi(t)}{t} = \infty.
\]
The Orlicz class $L_{\varphi}(\Omega)^2$ consists of all measurable functions $\mbu:\Omega\to\R^2$ such that
\[
\int_{\Omega} \varphi(|\mbu|) < \infty.
\] 
The Orlicz space $L_{\varphi^*}(\Omega)^2$ is the linear span of functions in $L_{\varphi}(\Omega)^2$ and it becomes a Banach space when equipped with the Luxemburg norm
\[
\|\mbu\|_{L_{\varphi^*}} = \inf\left\{k \ge 0\,:\,\int_{\Omega}\varphi(|\mbu|/k) \leq 1\right\}.
\]
\begin{Remark}\label{rem:orlicz}
    Let $E_{\varphi}(\Omega)^2$ denote the closure of $L^\infty(\Omega)^2$ in $L_{\varphi^*}(\Omega)^2$. 
    The convexity of $\varphi$ implies $E_{\varphi}(\Omega)^2 \subset L_{\varphi}(\Omega)^2$, see \cite{orlicz_book}. 
    Moreover, given a sequence $(\mbu_k)\subset L_{\varphi}(\Omega)^2$ and
     $\mbu\in L_{\varphi^*}(\Omega)^2$, we say that $\mbu_k$ is mean convergent to $\mbu$ if
    \[
    \int_{\Omega}\varphi(|\mbu_k-\mbu|)\,dx \to 0,\,\,k\to\infty.
    \]
    Norm convergence in $L_{\varphi^*}(\Omega)^2$ is stronger than mean convergence. 
\end{Remark}
Next, we equip the space  $\tV_h^{r}(\Omega)^2$ with the norn
\begin{align}
\norm{\mbw}^2_{H^1(\Omega, T_h)} = \norm{\mbw}^2_{L^2(\Omega)} + |\mbw|^2_{H^1(\Omega, T_h)},  
\end{align}
where
\begin{align}
|\mbw|^2_{H^1(\Omega, T_h)} =  \sum_{K \in T_h}\int_K | \nabla \mbw|^2 + \sum_{e \in E_h^i}
 \frac{1}{h_e}\int_e |\jumpop{ \mbw}|^2,  
 \end{align}
for all  $\mbw \in \tV_h^{r}(\Omega)^2$. Our strategy relies on proving an embedding
 theorem of the space $\tV_h^{r}(\Omega)^2$, 
for all $r\ge1$, into the Orlicz space $L_{\varphi^*}(\Omega)$ where
\[
\varphi(t) =  e^{|t|^2} - 1.
\]
This embedding is proved in Theorem~\ref{Orlicz_dg} below,
which extends Trudinger's embedding theorem for Orlicz spaces, \cite[Theorem 2]{trudinger1967imbeddings},
to the DG finite element setting. 



It will be useful to use  the reconstruction operator  of Karakashian and Pascal,  \cite{karakashian2003posteriori}. The next result is a local version of its approximation properties  established  in 
 \cite{karakashian2003posteriori}. 
\begin{Lemma} 
    \label{Lem:rec_bouds}
        Let $\omega_e(\varv)$ denote the set of edges that contain the node $\varv$, 
        i.e. $\omega_e(\varv) = \{e \in E^i_h | \varv \in e \}$.
        Then for $\mbu \in \tV_h^r (\Omega)^2$ there exists a reconstruction operator  
         $Q: \tV_h^r(\Omega)^2 \rightarrow V^r_h(\Omega)^2$ such that
    \begin{align}
       \sum_{K \in T_h}  \norm{\mbu -  Q\mbu}_{H^a(K)^2}^2 \le c  \sum_{e \in E_h^i} h_e^{1-2a} 
      \int_e |\jumpop{\mbu}|^2
      \label{rec_oper_error}
      \\
        \norm{\mbu - Q\mbu}_{L^\infty(K)^2}^2 \le c  \sum_{e \in \omega_e(K)} \frac{1}{h_e} 
\int_e |\jumpop{\mbu}|^2,
\label{rec_oper_Linferror_K}     
    \end{align}
    where $\omega_e(K) = \cup_{\varv \in K}\omega_e(\varv)$.
   \end{Lemma}

   \bigskip
   
 The proof requires an appropriate adaptation of  \cite[Theorem 2.1]{karakashian2004adaptive}, see also \cite{demlow2012pointwise}.  
   Let  $\omega_\varv = \{K_1, ..., K_{|\omega_\varv|} \}$, where every consequence pair share the edge 
   $e_i^\varv = K_i \cap K_{i+1}$, when $i=1, ..., |\omega_\varv| -1.$
%
Using the explicit definition of $Q,$ as an averaging operator,   \cite{karakashian2004adaptive}, appropriate  scaling arguments and the 
     inverse inequality 
   \begin{equation*}
   \begin{aligned}
  \norm{ \jumpop{\mbu} }_{L^\infty(e_i^\varv)}^2 
   \le \frac{c}{|e_i|}  \norm{ \jumpop{\mbu} }_{L^2(e_i^\varv)}^2, 
   \end{aligned} \nonumber
   \end{equation*} 
    the proof relies on establishing 
     \begin{equation}
     \begin{aligned}
      \norm{\mbu - Q\mbu}_{L^\infty(K)^2}^2 &
      \le 
       C\sum_{\varv \in K} \sum_{i=1}^{|\omega_\varv| -1 }  \frac{c}{|e_i^\varv |}  \norm{ \jumpop{\mbu} }_{L^2(e_i^\varv)}^2
     \
     \le C  \sum_{e \in \omega_e(K)} \frac{1}{h_e}  \int_e |\jumpop{\mbu}|^2. 
     \end{aligned}
     \end{equation}
For details see \cite {grekasPhD} and previous versions of this manuscript. The next result will be useful.

\begin{Lemma}
    \label{pointwise_bound}
   Let $\mbu \in \tV_h^r(\Omega)^2$,  for all $\mbx \in \Omega$ it holds
    \begin{align}
    |\mbu(\mbx)| \le C |\mbu|_{H^1(\Omega, T_h)} + |\hmbu(\mbx)|,
    \end{align}
    where $\hmbu = Q\mbu$ and $Q$ the reconstruction operator of Lemma~\ref{Lem:rec_bouds}.
    \begin{proof}
        For all  $\mbx \in \Omega$ we have 
        \begin{align}
        |\mbu(\mbx)| \le    |\mbu(\mbx) - \hmbu(\mbx) | + |\mbu(\mbx) | \le \norm{\mbu- \hmbu}_{L^\infty(\Omega)} 
        + |\hmbu(\mbx) |.
            \nonumber
        \end{align}
        $T_h$ is a triangulation of the domain $\Omega$, consequently there exists $K^*\in T_h$
        such that $\norm{\mbu- \hmbu}_{L^\infty(\Omega)} =   \norm{\mbu- \hmbu}_{L^\infty(K^*)}$.
        Using inequality (\ref{rec_oper_Linferror_K}) we deduce that
        \begin{align}
          \norm{\mbu- \hmbu}_{L^\infty(K^*)}^2 
        \le C \sum_{e \in E_h^i} \frac{1}{h_e}  \int_e |\jumpop{\mbu}|^2 \le C |\mbu|_{H^1(\Omega, T_h)}^2.
            \nonumber
        \end{align}      
        
    \end{proof}
\end{Lemma}

Next we state a crucial lemma, stated in  \cite[Lemma 1]{trudinger1967imbeddings}, for a complete proof of the embedding theorem.
\begin{Lemma}\label{lemma1}
    Let $\Omega \subset \R^2$  satisfy the cone condition and $\mbw\in H^1(\Omega)^2$. Then for a.e. $\mbx \in\Omega$ it holds that
    \[
    |\mbw( \mbx)| \leq C(\Omega) \left(\|\mbw\|_{L^1(\Omega)}+  
    \int_{\Omega}\frac{| \nabla \mbw(\mbxi) | }{|\mbx-\mbxi|}\,d\mbxi \right)
    \]
\end{Lemma}
 We may now prove the embedding. 

\begin{Theorem}
    \label{Orlicz_dg}
    Let $\Omega\subset R^2$ satisfy the cone condition. For each $h>0$, the space $\tV^r_h(\Omega)^2$ is continuously embedded into the Orlicz space 
    $L_{\varphi^*}(\Omega)$ where 
    \[
    \varphi(t) = e^{t^2} - 1.
    \]
    Furthermore, for any $\psi$ such that $\psi(t) \leq \varphi(\lambda t)$ for some $\lambda>0$, the space 
    $\tV^r_h(\Omega)^2$ is continuously embedded, in the sense of mean convergence, into the Orlicz class $L_{\psi}(\Omega)$; i.e.,
     whenever $\norm{\mbu_k - \mbu}_{H^1(\Omega, T_h)}\to 0$ then
    \[
    \int_{\Omega} \psi(\mbu_k - \mbu) \,d\mbx \to 0.
    \]
\end{Theorem}

\begin{proof}
    We exhibit the existence of constants $b=b(\mbu)>0$ and $C=C(\Omega)>0$ such that
    \begin{equation}\label{eq:appthm1_new}
    \int_{\Omega} e^{b|\mbu|^2} -1 \leq C(\Omega).
    \end{equation}
    To estimate the exponential of $|\mbu|^2$, we require to estimate all $L^q$ norms of $\mbu$ for $1\leq q < \infty$. 
    Since 
    \[
    \|\mbu\|_{L^q} = \sup_{f\in L^p} \frac{|\int_{\Omega}\mbu(\mbx)\cdot \mbf(\mbx)\,d\mbx|}
    {\|\mbf\|_{L^p}},
    \]
    using Lemma \ref{pointwise_bound}, we conclude
    \begin{equation}
    \begin{aligned}
    \int_{\Omega}|\mbf(\mbx)\mbu(\mbx)| 
    &\leq  \int_\Omega c |\mbf(\mbx)| |\mbu|_{H^1(\Omega, T_h)} +|\mbf(\mbx)| |\hmbu(\mbx)| 
    \\
    &\leq  C|\mbu|_{H^1(\Omega, T_h)}  |\Omega|^{1/q}\|\mbf\|_{L^p} 
   +  \int_\Omega |\mbf(\mbx)| |\hmbu(\mbx)|, 
    \label{eq:main_new}
    \end{aligned}
    \end{equation}
    where $\hmbu = Q \mbu$ and $\hmbu \in H^1(\Omega)^2$.
As in the  proof of \cite[Theorem 2]{trudinger1967imbeddings}
one can show that
    \begin{align}
     \int_{\Omega}|\hmbu(\mbx)| |\mbf(\mbx)|\,d\mbx
     \leq C C(\Omega)^{1/q} \norm{\hmbu}_{H^1(\Omega)} q^{1/2} \norm{\mbf}_{L^p(\Omega)}.
    \label{Lq_tu_bound}
    \end{align}
  For completeness we highlight some  details in the proof of  (\ref{Lq_tu_bound}) 
  following  \cite{trudinger1967imbeddings}.
Using Lemma \ref{lemma1}, we conclude
\begin{equation}
\begin{aligned}
\int_{\Omega}|\mbf(\mbx)\hmbu(\mbx)| 
&\leq C\left(\int_\Omega |\mbf(\mbx)|   \|\hmbu\|_{L^1} d\mbx+
\int_{\Omega}\int_\Omega\frac{|\mbf(\mbx)||\nabla \hmbu(\mbxi)|}{|\mbxi - \mbx|} d \mbxi d\mbx\right)
\\
&\le   C  \|\hmbu\|_{L^1}  |\Omega|^{1/q}\|\mbf\|_{L^p} 
+ C\int_{\Omega}\int_\Omega\frac{|\mbf(\mbx)||\nabla \hmbu(\mbxi)|}{|\mbxi - \mbx|} d \mbxi d\mbx.
\label{eq:main_old}
\end{aligned}
\end{equation}
Regarding the double integral above, applying  the Cauchy-Schwarz inequality gives 
\begin{align}
\label{eq:0_old}
\int_{\Omega}\int_\Omega &\frac{|\mbf( \mbx)| |\nabla \hmbu(\mbxi)|}{|\mbxi - \mbx|} d\mbxi d\mbx  
 \leq \left(\int_{\Omega}\int_\Omega\frac{|\mbf(\mbx)|}{|\mbx-\mbxi|^{2-1/q}}\right)^{1/2}
\left(\int_\Omega\int_\Omega\frac{|\nabla \hmbu(\mbxi)|^2|f(\mbx)|}{|\mbx-\mbxi |^{1/q}}\right)^{1/2}.
\end{align}
We estimate the two double integrals separately. Denoting by $d$ the diameter of $\Omega$, 
we have
\[
\int_{\Omega}\frac{1}{|\mbx-\mbxi|^{2-1/q}} d\mbxi \leq 
\int_{B_d(0)}| \mby|^{-2+1/q} \leq C \int_0^d r^{-2+1/q} r\, dr = C d^{1/q} q.
\]
Hence,
\[
\int_{\Omega}\int_\Omega\frac{|\mbf(\mbx)|}{| \mbx-\mbxi|^{2-1/q}} 
\leq C q d^{1/q} \int_\Omega |\mbf(\mbx)| \,d\mbx \leq C d^{1/q} 
|\Omega|^{1/q} \|\mbf\|_{L^p} q \leq C d^{3/q}  \|\mbf\|_{L^p} q.
\]
On the other hand, we find that
\[
\int_\Omega\frac{|\mbf(\mbx)|}{|\mbx-\mbxi |^{1/q}}\,d\mbx 
\leq \|\mbf \|_{L^p}\left(\int_{\Omega}| \mbx-\mbxi|^{-1}\right)^{1/q} 
\]
and, therefore as before,
\begin{equation}\label{eq:1a_old}
\int_\Omega\frac{|\mbf(\mbx)|}{| \mbx-\mbxi |^{1/q}}\,d\mbx \leq C^{1/q}d^{1/q} \|\mbf\|_{L^p}.
\end{equation}
Then, the second term in \eqref{eq:0_old} becomes
\begin{equation}\label{eq:2_old}
\int_\Omega\int_\Omega\frac{|\nabla \hmbu(\mbxi)|^2|\mbf(\mbx)|}
{|\mbx-\mbxi |^{1/q}} \leq C^{1/q}d^{1/q}\|\nabla \hmbu\|^2_{L^2} \|\mbf\|_{L^p}.
\end{equation}
Combining \eqref{eq:1a_old}-\eqref{eq:2_old} and \eqref{eq:0_old}, we deduce that
\[
\int_{\Omega}\int_\Omega\frac{|\mbf(\mbx)||\nabla \hmbu(\mbxi)|}{|\mbxi - \mbx|} 
\leq C C^{1/q}d^{2/q} \|\nabla \hmbu\|_{L^2} \|\mbf\|_{L^p} q^{1/2}.
\]    
Replacing the above bound in (\ref{eq:main_old}) we obtain (\ref{Lq_tu_bound}).

Our aim is  to bound (\ref{eq:main_new}) with respect to $\norm{\mbu}_{H^1(\Omega, T_h)}$.     
From (\ref{rec_oper_error})  the following bound holds 
\begin{equation}
\begin{aligned}
\norm{\hmbu|}^2_{H^1(\Omega)} &\le 2\sum_{K \in T_h}\norm{\hmbu - \mbu|}^2_{H^1(K)} +  2\sum_{K \in T_h}\norm{ \mbu}^2_{H^1(K)}
\\
&\le 
c \left(  \sum_{e \in E_h^i} h_e^{-1} \int_e |\jumpop{\mbu}|^2 + \sum_{K \in T_h}| \mbu|^2_{H^1(K)} + \norm{ \mbu}^2_{L^2(\Omega) } \right)
\\
&\le c \norm{\mbu}^2_{H^1(\Omega, T_h)}
      \nonumber
\end{aligned}
\end{equation}
Therefore  equation (\ref{Lq_tu_bound}) becomes
    \begin{align}
\int_{\Omega}|\hmbu(\mbx)| |\mbf(\mbx)|\,d\mbx
\leq c C(\Omega)^{1/q} \norm{\mbu}_{H^1(\Omega, T_h)} q^{1/2} \norm{\mbf}_{L^p(\Omega)}.
\label{Lq_u_bound}
\end{align}
Returning to (\ref{eq:main_new}) we infer that 
\[
    \int_{\Omega}|\mbf(\mbx)\mbu(\mbx)|  \le C \norm{\mbu}_{H^1(\Omega, T_h)} q^{1/2} \norm{\mbf}_{L^p(\Omega)},
\]    
    leading to the estimate
    \[
    \|\mbu\|_{L^q}\leq C C(\Omega)^{1/q} \norm{\mbu}_{H^1(\Omega, T_h)} q^{1/2}.
    \]
     In particular, note that
    \[
    \int_\Omega |\mbu|^{2q} \leq C(\Omega) \left(C \norm{\mbu}^2_{H^1(\Omega, T_h)}q\right)^q,
    \]
    so that, choosing $b>0$ such that $b C\norm{\mbu}^2_{H^1(\Omega, T_h)} <1/e$ we reach \eqref{eq:appthm1_new}.
    
    Regarding the embedding in the sense of mean convergence, as in \cite{trudinger1967imbeddings}, 
    we note that bounded functions are dense in the space $\tV^r_h(\Omega)^2$ and hence
     $\tV^r_h(\Omega)^2\subset E_{\varphi}(\Omega)$ due to the above embedding. In particular,
      if $\psi(t)\leq \varphi(\lambda t)$, for some $\lambda>0$,  $\tV^r_h(\Omega)\subset L_{\psi}(\Omega)$
       and the embedding is continuous with respect to mean convergence as, by Remark \ref{rem:orlicz}, norm convergence implies convergence in the mean.
\end{proof}

\begin{prop}\label{prop:appendix}
Let $\Phi:\R^{2\times 2}\to\R$ a continuous function satisfying
\[
|\Phi(\mbxi)| \le c_1e^{c_2 |\mbxi|^2}, \quad \forall \mbxi \in \R^{2 \times 2}
\]
and suppose that
\[ \norm{\mbu_h - \mbu}_{H^1(\Omega)} + \tred{|\nabla \mbu_h - \nabla \mbu|_{H^1(\Omega, T_h)} } \rightarrow 0, 
\quad h \rightarrow 0,
\]
for $\mbu_h \in V^q_h(\Omega)^2$  
and $\mbu \in H^2(\Omega)^2$. 
Then, up to extracting a subsequence,
\begin{align}
\int_\Omega \Phi(\nabla \mbu_{h}) \rightarrow \int_\Omega \Phi(\nabla \mbu), 
\quad \text{as } h \rightarrow 0.
\end{align}
\end{prop}
 
 \begin{proof}
Let $\mbw_h = \nabla \mbu_h$, $\mbw=\nabla \mbu$, then
$ \|\mbw_h - \mbw\|_{L^2(\Omega)} + |\mbw_h - \mbw|_{H^1(\Omega,T_h)} \to 0,\,\,h\to0.$
Up to extracting a subsequence, $\mbw_h\to \mbw$ pointwise a.e. and, by the continuity of
 $\Phi$, also $\Phi(\mbw_h) \to \Phi(\mbw)$ a.e. in $\Omega$.

Next, note that
\begin{align}\label{eq:propapp1}
|\Phi(\mbw_h)|  \leq c_1\left(e^{c_2|\mbw_h|^2} - 1\right) + c_1 = c_1 \psi(\mbw_h/2) +c_1,
\end{align}
where $\psi(t) = e^{4c_2|t|^2}-1$, is a convex function such that $\psi(t) = \varphi(\lambda t)$ for 
$\lambda=4c_2$, where $\phi$ is given in Theorem~\ref{Orlicz_dg}. The convexity of $\psi$ implies that
\[
\psi(\mbw_h/2)\leq \frac12 \psi(\mbw_h - \mbw) + \frac12\psi(\mbw).
\]
By Theorem~\ref{Orlicz_dg}, we have that $(\psi(\mbw_h - \mbw))$ converges in $L^1(\Omega)$, as $h\rightarrow 0$.
Therefore $\psi(\mbw_h/2)$ is uniformly integrable. But then \eqref{eq:propapp1} implies that 
$(\Phi(\mbw_h))$ is a uniformly integrable sequence and thus $\Phi(\mbw_h) \to \Phi(\mbw)$ in $L^1(\Omega)$.
 \end{proof}

\begin{Remark}
\label{Remark:cOrlicz}
Proposition \ref{prop:appendix} and the classical embedding of Trudinger for Orlicz spaces \cite{trudinger1967imbeddings}, 
also implies that if $(\mbu_\delta) \subset H^2(\Omega)$ and
$\norm{\mbu_\delta - \mbu}_{H^2(\Omega)} \rightarrow 0$, as $\delta \rightarrow 0$ then
$$\int_\Omega \Phi(\nabla \mbu_{\delta}) \rightarrow \int_\Omega \Phi(\nabla \mbu), \quad 
\text{as } \delta \rightarrow 0.$$ 
\end{Remark}


Finally, we establish  the $\Gamma$-convergence and the convergence of discrete absolute
 minimizers when the penalty term $\Phi$ has exponential growth.

\begin{Theorem}
\label{Thm:Gamma_conv_exp}
Let the penalty function $\Phi$ have the exponential growth (\ref{pen:exponential_growth}).
Then, the following properties hold:
\begin{enumerate}[\labelsep=-0.2em (i)]
\item  for all $\mbu \in\mathbb{A}(\Omega)$, there exists a sequence
   $(\mbu_h)_{h>0}$ with $\mbu_h \in \mathbb{A}^q_h(\Omega)$,  such that 
$\mbu_h \rightarrow \mbu$ in $H^1(\Omega)^2$ and 
\begin{align}
 \Psi[\mbu] \geq \limsup\limits_{h\rightarrow 0} \Psi_h[\mbu_h];
\end{align}
\label{limsup_exp}

\item \vspace{-0.1cm} for all $\mbu \in \mathbb{A}(\Omega)$ and all sequences 
$(\mbu_h)\subset \mathbb{A}^q_h(\Omega)$
 such that $\mbu_h \rightarrow \mbu$ in $H^1(\Omega)$  it holds
 \begin{align}
 \Psi[\mbu] \leq \liminf\limits_{h\rightarrow0} \Psi_h[\mbu_h];
\end{align}
\label{liminf_exp}
\item \vspace{-0.1cm} Theorem~\ref{Thm:min_convergence} holds, i.e., the discrete  absolute minimizers
 converge to an absolute minimizer of the continuous problem.

%
\end{enumerate}

\end{Theorem}
\begin{proof}

(i) Following the proof of  Theorem \ref{Theorem:limsup} it remains only  to verify
  $\Phi(\nabla \mbu_{h, \delta_h}) \rightarrow \Phi(\nabla \mbu)$ in $L^1(\Omega)$, from
  (\ref{limsup:Phi_convergence}),
where $(\mbu_{h, \delta_h})$ is the sequence defined after (\ref{ud_conv}). It is enough to show
 \begin{align}
 \Phi_{h}(\nabla \mbu_{h, \delta_h}) - \Phi(\nabla \mbu_{\delta_h}) \rightarrow 0 
 \quad \text{and}  \quad 
 \Phi(\nabla \mbu_{\delta_h}) - \Phi(\nabla \mbu) \rightarrow 0
 \end{align}
in $ L^1(\Omega)$, as $h\rightarrow 0$. By Proposition \ref{prop:appendix},
it suffices  to show that 
 $\tred{|\nabla \mbu_{h, \delta_h} - \nabla \mbu_{ \delta_h}|_{H^1(\Omega, T_h)}} \rightarrow 0,
 \text{ as } h\rightarrow 0,$ i.e.,
\begin{align}
\sum_{K \in T_{h}} \int_K \left| \nabla \nabla \mbu_{{h, \delta_h}} - \nabla \nabla\mbu_{ \delta_h} \right|^2
+
\sum_{e \in E_{h}^i} \frac{1}{h_e}\int_e \left|\jumpop{\nabla \mbu_{{h, \delta_h}} - \nabla\mbu_{ \delta_h} }   \right|^2
\rightarrow 0,
\label{exp:h2_conv}
\end{align}
as $h \rightarrow 0$. Using (\ref{eq:uhd_error}),   (\ref{ud_conv})  (\ref{delta2h})
we obtain the following bounds for the first term
\begin{equation}
\begin{aligned}
\sum_{K \in T_{h}} \int_K \left| \nabla \nabla \mbu_{{h, \delta_h}} - \nabla \nabla\mbu_{ \delta_h} \right|^2 
&\le c h^2 \sum_{K \in T_{h}} \int_K |\mbu_{ \delta_h} |^2_{H^3(K)} 
 {\le  C h |\mbu|^2_{H^2(\Omega)}. }
\end{aligned}
\label{expert:limsup_j1}
\end{equation}
Note that  $\mbu_{\delta} \in H^3(\Omega)^{2 } $ implies 
$\mbu_{\delta} \in C^1(\Omega)^2$ from the Sobolev embedding. 
Then working as (\ref{limsup:jump_conv}) 
\begin{align}
\sum_{e \in E_{h}^i} \int_e \left|\jumpop{\nabla \mbu_{{h, \delta_h}} - \nabla\mbu_{ \delta_h} }   \right|^2 =
\sum_{e \in E_{h}^i} \frac{1}{h_e}\int_{e} |\jumpop{\nabla  \mbu_{h, \delta_h}} |^2
\le c h|\mbu|^2_{H^2(\Omega)}.
\label{expert:limsup_j2}
\end{align}
From (\ref{expert:limsup_j1}) and (\ref{expert:limsup_j2}) we deduce (\ref{exp:h2_conv}).
Also Remark \ref{Remark:cOrlicz} implies   $\Phi(\nabla \mbu_{\delta_h}) \rightarrow \Phi(\nabla \mbu)$, in 
$L^1(\Omega)$, which concludes the proof of (i).

\vspace{1 mm}
(ii) This is immediate as $\Phi$ is continuous  and thus by Fatou's Lemma,
\begin{align}
 \int_\Omega \Phi(\nabla \mbu) \le \liminf_{h \rightarrow 0} \int_\Omega \Phi(\nabla \mbu_h),
 \quad \text{when } \mbu_h \rightarrow \mbu \text{ in } H^1(\Omega).
\end{align}

\vspace{1 mm}
(iii)
From the $\Gamma-$convergence result, (\ref{limsup_exp}) and (\ref{liminf_exp}) , 
the proof of Theorem \ref{Thm:min_convergence} can be adopted.  
 \end{proof}

%

\section{Model and Computational Experiments}
\label{sec:computations}
{}

\noindent
\tblue{\emph{Numerical Energy Minimization.}} The potential energy of the continuous model has been discretized using the finite element
discretization provided by the 
{\fontfamily{ptm}\selectfont FEniCS project}
\cite{AlnaesBlechta2015a}. Specifically, we have used quadratic Lagrange elements,
i.e., $\mbu_h \in V^2_h(\Omega)^2$.
For the energy minimization procedure we have employed  the nonlinear conjugate
gradient method, see \cite{NoceWrig06}. A parallelization of the
nonlinear conjugate gradient algorithm has been developed using 
{\fontfamily{ptm}\selectfont petsc4py} data structures 
\cite{dalcin2011parallel, petsc-user-ref}. The resulting deformations are visualized with 
{\fontfamily{ptm}\selectfont paraview} \cite{squillacote2007paraview}.

The nonlinear conjugate gradient method is a popular method for large-scale nonlinear optimization 
problems. It is a variant of  the conventional conjugate gradient method, where for a quadratic 
function the conjugate directions and the minimum along a given direction can be computed 
explicitly. For the nonlinear variant, in the $k-$th iteration, given the direction $p_k$, the next search 
direction $p_{k+1}$ (conjugate direction in the conventional method) is computed through the 
Polar-Ribi\`ere method. The minimum value over a given direction is approximated employing line search 
algorithms that satisfy the strong Wolfe condition, which guarantees in certain 
circumstances that the computed direction $p_k$ is a descent direction. For details see 
\cite{NoceWrig06}.

\noindent
\emph{Strain Energy Density.} Here we describe our model for the strain energy density $W+\Phi$ in \eqref{equ:total_potential} by following \cite[Section 5.2.1]{Grekas_2021}.
Starting with a  (microscopic) fiber, we let the effective stretch $\lambda$ equal the distance between its endpoints divided by its undeformed (relaxed) length.  The energy of a single fiber can be expressed  as a function   $w(\lambda)$ of the effective stretch $\lambda$.  When the fiber is in tension, it is straight and $\lambda$ equals the actual stretch (strain $+1$), while $w(\lambda)$ equals the elastic energy due to stretching of the fiber. While in compression, it may be buckled; in which case the elastic (mostly bending) energy of the fiber can still be expressed as a function  $w(\lambda)$ of the distance between its endpoints.  In order to model a 1D fiber energy  $w(\lambda)$  for a single fiber that buckles in compression, we start with the derivative $S(\la)= w'(\lambda)$, which represents force as a function of stretch. We choose a  polynomial that has increasing slope so that $S'(\lambda)<S'(1)$ for $0<\lambda<1$ (softening in compression due to buckling, and $S'(\lambda)>S'(1)$ for $\lambda>1$ (stiffening in tension). An example is $S(\lambda)=\mu (\lambda^5-\lambda^3)$ for $\lambda>0$ with $\mu =$const.$>0$.
Integrating this with respect to $\lambda$  gives an energy
\begin{equation}
w(\la)= \mu ( \la^6/6 -\la^4/4).
  \label{wa}
\end{equation}
We model the ECM as a 2D nonlinear elastic continuum undergoing  deformations $ \mby(\mbx)$ where a particle with position vector $\mbx$ in the undeformed state is mapped to deformed position $\mby=\mby(\mbx)=\mbx+\mbu(\mbx)$, with $\mbu$ the displacement.  The strain energy density of the material can be written as a function $\tW(\mbF)$ of the deformation gradient $\mbF=\nabla\mby$.  We model the ECM as an isotropic material, which means that $W$ depends on $\mbF$ only through the principal stretches $\la_1$, $\la_2$,  the eigenvalues of the  right stretch tensor $(\mbF^T\mbF)^{1/2}$. 
To connect the single fiber energy with the 2D strain energy density $W$ we follow e.g., \cite{vainchtein2002strain}.  We suppose the ECM consists of uniformly distributed fibers at the microscopic level. A macroscopically affine deformation is equivalent to a biaxial stretch in two orthogonal directions with stretches  $\lambda_1$, $\lambda_2$, modulo rotation.  A fiber of undeformed length $l$ making an angle $\theta$ with the  stretch axes in the undeformed state,  will have endpoints at $(0,0)$ and $(l\cos\theta,l\sin\theta)$. After deformation the latter is $(\lambda_1l\cos\theta,\la_2l\sin\theta)$. As a result, the effective stretch  of the fiber is $\lambda=\lambda(\theta)=\sqrt{(\lambda_1 \cos\theta)^2 + (\lambda_2 \sin\theta)^2}$, and its energy is $w(\lambda(\theta))$.  Summing over all fiber orientation angles $\theta$, we obtain the macroscopic elastic energy density of the fibrous ECM:
$$ \hat{W}(\lambda_1, \la_2) = \frac{1}{2\pi}\int_0^{2\pi} w(\lambda(\theta)) d\theta= \frac{1}{2\pi}\int_0^{2\pi} w \left(\sqrt{(\lambda_1 \cos\theta)^2
 + (\lambda_2 \sin\theta)^2} \right) d \theta.$$  
For $w(\lambda)$ given e.g. by (\ref{wa}), this integral can be evaluated explicitly:
\begin{align}
W(\nabla \mbu) = \tW(\mbF) =  \frac{\mu}{96} ( 5I_1^3 - 9I_1^2 -12 I_1 J^2 + 12 J^2 +8),
\label{equ:en_fun}
\end{align}
Here the deformation invariants are $I_1=I_1(\mbF)=\hbox{tr}(\mbF^T\mbF)=\la_1^2+\la_2^2$ and $J=J(\mbF)=\det\mbF=\la_1\la_2$, the Jacobian determinant of the deformation. The ratio of deformed to undeformed density equals $1/J$. From \cite[Lemma A.1.1]{grekasPhD}, $W$ satisfies the lower bound of (\ref{W:properties}). 
For an upper bound one has to remove the negative terms and use the 
inequality $2J \le I_1$.
We  add a fiber volume penalty term to the energy to account for resistance of  densified fibers to complete crushing by virtue of their nonzero volume, and to penalize intepenetration of matter. This term increases the energy abruptly when the Jacobian  $J$ becomes less than a small positive constant $b<<1$, while it becomes negligible as $J$ increases from $b$.  An example is 
\begin{equation}\tilde{\Phi}(J)=\exp[A(b- J)]  \label{Phi}
\end{equation}
with $\Phi(\nabla\mbu)=\tilde{\Phi}(\det({\bm 1} + \nabla \mbu))$, where $A$ is a large positive constant. 
The strain energy function $W  + \Phi$
 is  non rank-one convex. In fact the total energy density expressed as a function of the principal stretch pair $U(\la_1,\la_2)$ is
a double-well potential, modulo a null Lagrangian  \cite{Grekas_2021}, which can be chosen so that the  two minima of $U$ are of zero energy which is positive elsewhere. For  a typical choice of parameters, the two minima (wells) of $U$ are the reference state $(\la_1,\la_2)=(1,1)$ and the state $(\la_1,\la_2)=(0.2,1.06)$. The latter is a severe compression in one direction ($\la_1=0.2<1$) combined with a moderate stretch ($\la_2=1.06>1$) in an orthogonal direction. This compressed well involves an almost fivefold density increase and corresponds to the densified phase consisting of buckled, collapsed fibers. In our simulations, the deformation in  tethers and hairs is in the neighborhood of the compressed, densified well, while outside them it is close to the undeformed well.


\noindent
\emph{Simulations.} In the experiments of \cite{Grekas_2021}, the rather unpredictable live cells were replaced by round active particles which 
were embedded in the ECM. These  contract on demand by a 50\% decrease in radius, thereby exerting tractions onto the surrounding ECM that trigger the 
phase microstructure formation.
In our simulations, the ECM is modelled as a homogeneous material with the strain energy function just described. In the undeformed (reference) 
configuration, the ECM occupies the part $\Omega$ of a rectangular or round domain exterior to one or more circles $C_i$ of radius $r_c$, which represent
active-particle boundaries. Contraction of these particles is modelled by imposing Dirichlet boundary conditions on the displacement $\mbu$. The 
displacement field $\mbu$ is required to vanish on the outside boundary, which is thus assumed fixed. At the inner boundaries $C_i$, where 
$|\mbx-\mbz_i|=r_c$ (circles of radius $r_c$ and center $\mbz_i$) $\mbu$ is specified as 
\begin{equation}
    \mbu (\mbx)=-u_0(\mbx-\mbz_i),  \quad   \mbx\in C_i.\label{bccc}
\end{equation}
This represents a radial contraction which maps each particle boundary to a smaller circle of radius $r'_c=r_c-u_0$. The constant $u_0$ is obtained from   experimental data and has a typical value of 0.5$r_c$. For a more sophisticated model of active particles that allows shape deviations and motion of centers due to deformation see Remark~\ref{alternative_bc} and \cite[Section 2.4, Section 5.2.2]{Grekas_2021}. 
Computations involving a single active particle contracted by $50\%$ at the center of the ECM domain are shown in Figure~\ref{fig:mesh_dependence}, with a color map of the ratio of deformed over reference density.
The densified phase (red) occurs in radial hairlike bands that taper off into the undensified phase (blue). 
Figures~\ref{fig:h0/4} and \ref{fig:h0/8} are examples of  fine phase mixtures. The mixture of low and high strain phases
 is energetically preferable because the average strain it produces is compatible with the Dirichlet boundary conditions, whereas the  strains in the densified phase, despite their low energy, are not compatible with the boundary deformation.

In the case  of two contracting active particles depicted in Figure~\ref{fig:2cells}, the material between the particles is stretched along the axis passing
through the active particles' centers and is compressed in the transverse direction. This renders the strain state of the densified energy well energetically favorable.
As  depicted in Figure~\ref{fig:2cells}, above some critical value of
active particle contraction a  tract in the densified phase, namely a tether between  particles emerges, while 
radial bands  
emanate from each particle boundary as in the case of a single particle. 

\tblue{Examples of the  agreement  between simulations and 
experiments are shown in Figure~\ref{fig:exp_comparison}.  In the   
experiments of \cite{Grekas_2021},a tether is sometimes observed to split 
into thinner parallel bands (Figure~\ref{fig:fn3a}). 
This is similar to the phenomenon of twinned martensite in crystals, namely 
 splitting and tapering of twin bands in
a crystal near an incompatible boundary \cite{james1995modeling}. Here
as well, energy minimization forces strains to stay close to energy-density minima. The active particles in Figure~\ref{fig:fn3a}
contracted by $u_0/r_c$ = 38\%. The azimuthal stretch 
$\lambda_{\theta} = 1 -u_0/r_c=62\% $ imposed at the particle boundary by
contraction is incompatible with the stretch $\la_1=20\%$ corresponding
to the densified-phase energy well. To avoid this
mismatch while maintaining displacement continuity, the
tethers splits into narrow bands to minimize contact with 
the particle boundary (experiment:Figure~\ref{fig:fn3a}, simulation: Figure~\ref{fig:fn3b}).}

\tblue{
Can our model predict  how close  particles should be and how much should they contract in order for a tether to form betwewen them? Particle contraction and
distance between particles have been varied  in multiple simulations in \cite{Grekas_2021}. 
This provided a separatrix
curve of average particle radial strain versus distance
between particles (blue curve in Figure~\ref{fig:fn3c}). Above this
curve, our model predicts that a tether forms joining the
two particles; below the curve no tether will form. Data
from our experiments agreed with this prediction: blue
points in Figure~\ref{fig:fn3c} are data points from experimental particle pairs with
a tether observed joining them, red points correspond to
pairs without a tether between them.
Furthermore, our model predicts that the correct displacement decay rate with distance from a single contracting particle,  Figure~\ref{fig:fn3d}, which is much slower than in materials that do not suffer phase transition. This is an
important feature for long-range mechanosensing. Experiments with
fibroblasts \cite{Notbohm2015}, reveal the same displacement decay with distance as the one predicted by our computations, of the form $r^{-0.5}$. 
This means that the displacement fields propagate over a longer range compared to
the range predicted by linear elasticity, where decay is proportional to $r^{-2}$. For more details see \cite{grekasPhD}.}

If the regularization (higher gradient) term is omitted, i.e. $\varepsilon =0$, then the computed solutions 
depend on the mesh size; similar results can be found in \cite{luskin1996computation}. 
As  depicted in Figures~\ref{fig:h0}-\ref{fig:h0/8},
mesh resolution must be fine enough to capture localized deformations, but
further increases of resolution result in more and thinner bands around the active particle. 
The appearance of phase boundaries is due to  the ellipticity failure \cite{knowles1978,rosakis1993}
and the rank-one connected minima. 
The appearance of finer and finer phase mixtures as resolution is increased
is related to incompatibility of the wells with the boundary conditions\cite{james1995modeling}.  
As a result, the minimum is not attained, but minimizing sequences develop more and finer oscillations in order
to create less incompatible deformations of lower energy \cite{Ball1987}.\tblue{This is responsible for the mesh dependence, as increasing mesh resolution simply captures terms further along such minimising sequences as in Figures~\ref{fig:h0}-\ref{fig:h0/8}.}
The higher gradient 
term restores the ellipticity of the Euler-Lagrange equation, consequently regularizing the solution, ensuring the existence of a minimizer of limited fineness, and eliminating 
mesh dependence. 
In addition,  $\varepsilon$ can be considered  an internal length scale, controlling the thickness of transition  layers that replace gradient discontinuities.  This
means that smaller values of $\varepsilon$ permit finer microstructures with more and thinner hairs; see an example 
with varying $\varepsilon$ in Figure~\ref{fig:varying_e}.
\tblue{Figures~\ref{fig:7h0}-\ref{fig:7h0/8} illustrate that for  fixed $\varepsilon>0$ numerical solutions converge as mesh size tends to zero, as
expected from the analysis in this work, although it is possible that the limit state is merely an isolated local minimum of the energy in this
numerical example. The analysis in this work could conceivably apply to the case of convergence to an isolated local minimum as mesh size
approaches zero in the presence of  fixed $\varepsilon>0$ in the sense of \cite{braides2014local}; see also \cite{bartels2017bilayer}.
When $\varepsilon=0$, the energy functional in not lower semicontinuous, but for any $\varepsilon>0$ the
presented theory holds. In the numerical minimization, the local mesh size should be smaller than the length scale imposed by  $\varepsilon$ in
order to resolve fine structure as is illustrated by Figures~\ref{fig:7h0}-\ref{fig:7h0/8}. Therefore, for very small
values of $\varepsilon$ the degrees of freedom increase substantially due to the necessity of very high resolution, which may raise practical
issues  with the discrete minimization process.}

More complex cases include the contraction of multiple active particles. Then the
densified tehters connecting two cells can appear or disappear, influenced 
by other neighboring cells. An example can be seen in Figure~\ref{fig:many_cells}.

 In \cite{Grekas_2021} extensive simulations of the model are performed which 
exhibit excellent agreement with experimental results. \tblue{In particular the elastic phase transition model, combined with the numerical scheme analysed here is capable of predicting and explaining intricate details of the geometry of observed multiphase microstructures in fibrous collagen biomaterials.}

\begin{figure}[tbhp]
\centering
 \subfloat[]{\includegraphics[width=0.16\linewidth]{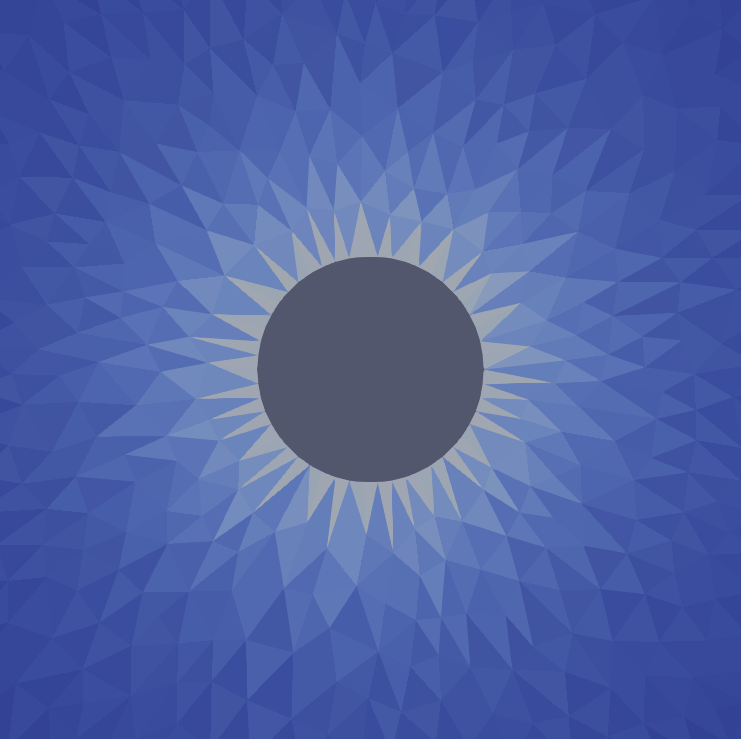} 
\label{fig:h0} }
\subfloat[]{ \includegraphics[width=0.16\linewidth]{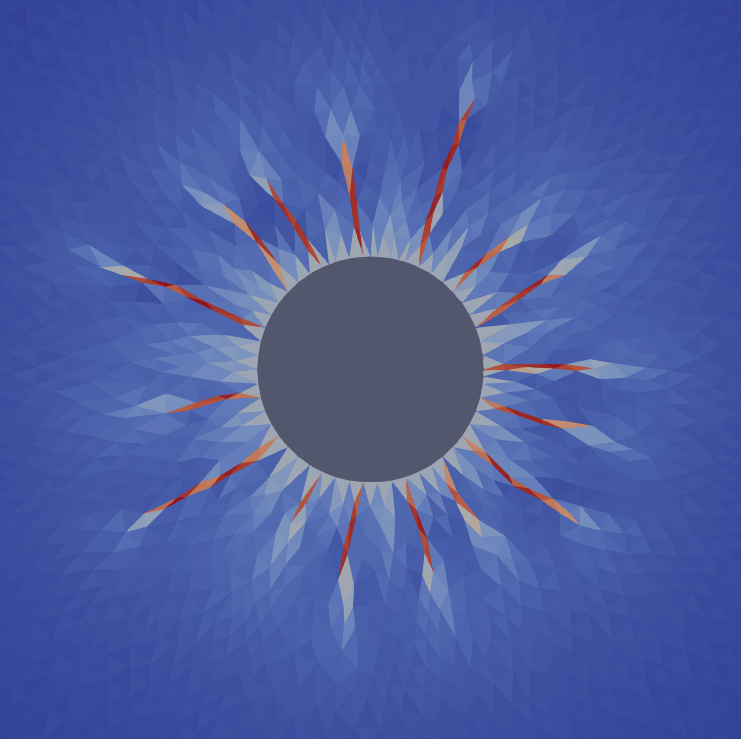}
\label{fig:h0/2} }
 \subfloat[]{\includegraphics[width=0.16\linewidth]{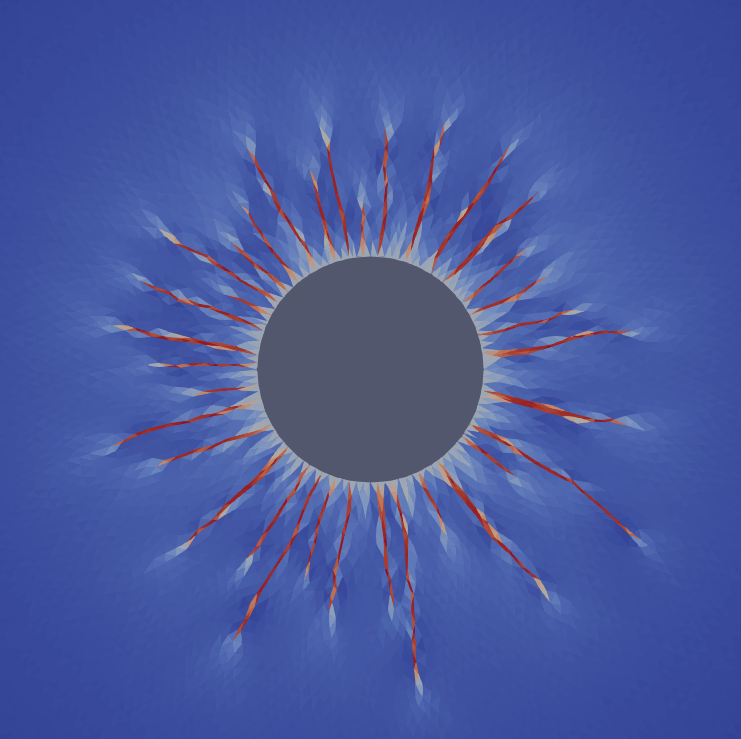} 
\label{fig:h0/4} }
\subfloat[]{ \includegraphics[width=0.16\linewidth]{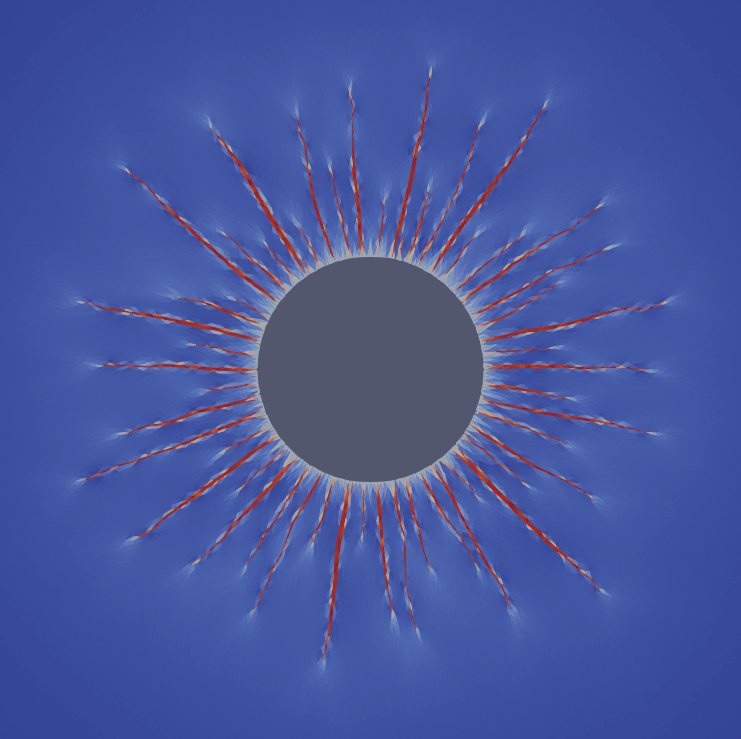}
 \label{fig:h0/8} }
 \includegraphics[height=2.4cm]{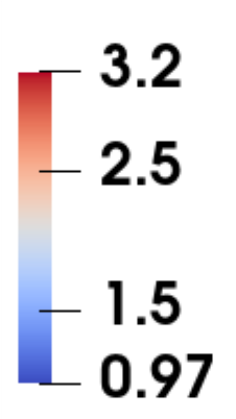}
 \\
\vspace{-0.4cm}
\subfloat[]{\includegraphics[width=0.16\linewidth]{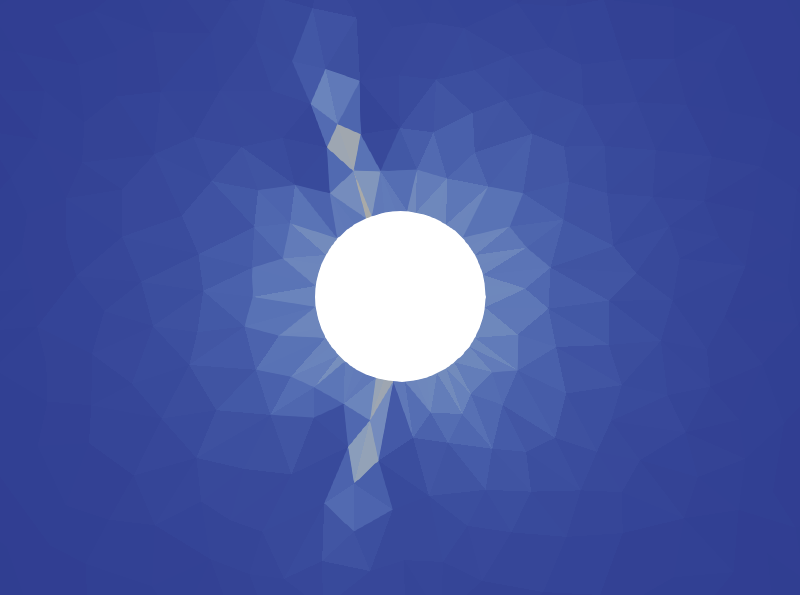} 
\label{fig:7h0} }
 \subfloat[]{\includegraphics[width=0.16\linewidth]{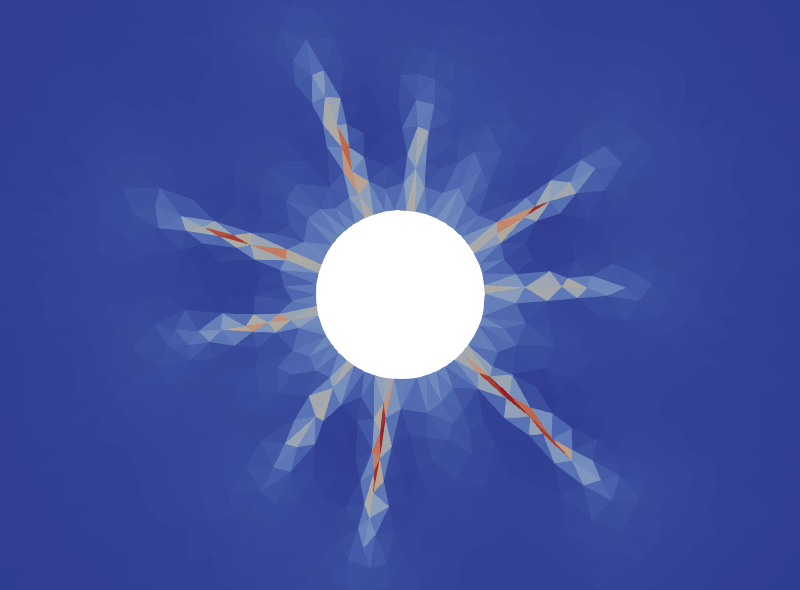} 
\label{fig:7h0/2} }
\subfloat[]{ \includegraphics[width=0.16\linewidth]{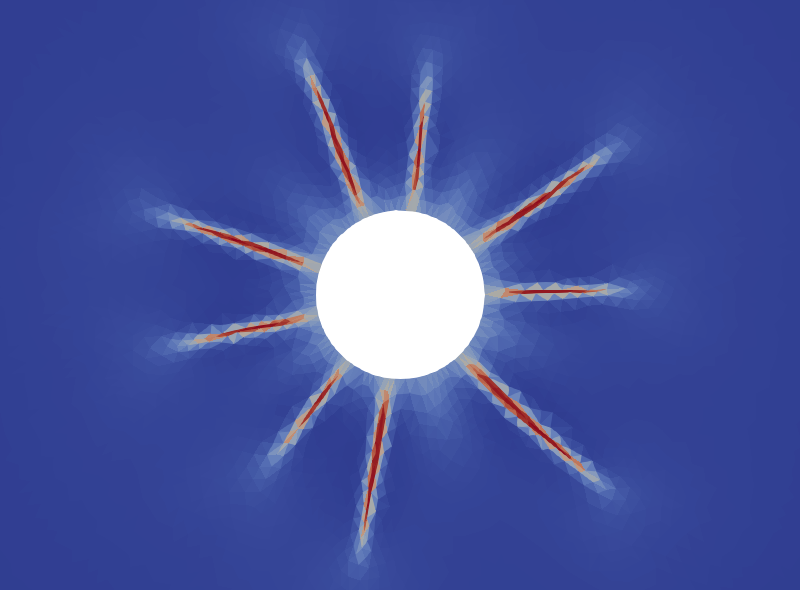}
\label{fig:7h0/4} }
 \subfloat[]{\includegraphics[width=0.16\linewidth]{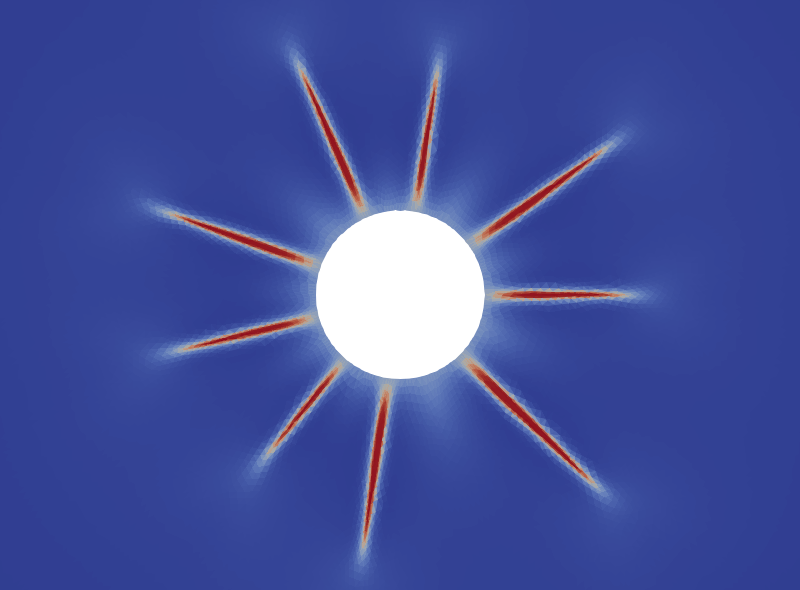} 
\label{fig:7h0/8} }
 \includegraphics[height=2.4cm]{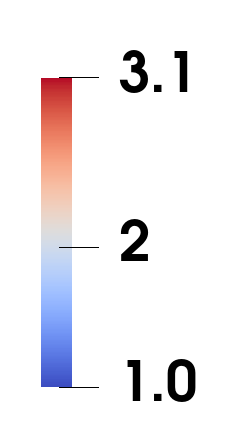} 
 \caption{ Matrix deformation for a $50\%$ contracting 
 active particle of undeformed radius $r_c$.  \hfil\break
  \protect\subref{fig:h0} -\protect\subref{fig:h0/8}: Excluding higher gradients, i.e. $\varepsilon=0$ while
 increasing mesh resolution. In \protect\subref{fig:h0} mesh size is $h_0\approx r_c/7$ 
 and microstructures are too thin compared to $h_0$  and cannot be captured. 
 Increasing mesh resolution to $h_0/2$, $h_0/4$ and $h_0/8$ in \protect\subref{fig:h0/2}, 
\protect\subref{fig:h0/4} and \protect\subref{fig:h0/8}, more and thinner hairlike microstructures emerge. \hfil\break
\protect\subref{fig:7h0} -\protect\subref{fig:7h0/8}: 
 Convergence of numerical solutions to a microstructure of finite fineness in the presence of fixed higher 
 gradient coefficient $\varepsilon=0.01 r_c$ with increasing mesh resolution.
 In \protect\subref{fig:7h0} mesh size is $h_0\approx r_c/4$;  
 microstructures are thinner than $h_0$ and cannot be captured.
Mesh resolution is increased to $h_0/2$, $h_0/4$, $h_0/8$  
in  \protect\subref{fig:7h0/2}, \protect\subref{fig:h0/4},\protect\subref{fig:h0/8} respectively. 
Last two panels are virtually identical.
 }
 \centering
 \label{fig:mesh_dependence}
\end{figure}

\vspace{-0.5cm}
\begin{figure}[tbhp]
  \centering
  \subfloat[]{\includegraphics[width=4.2cm]{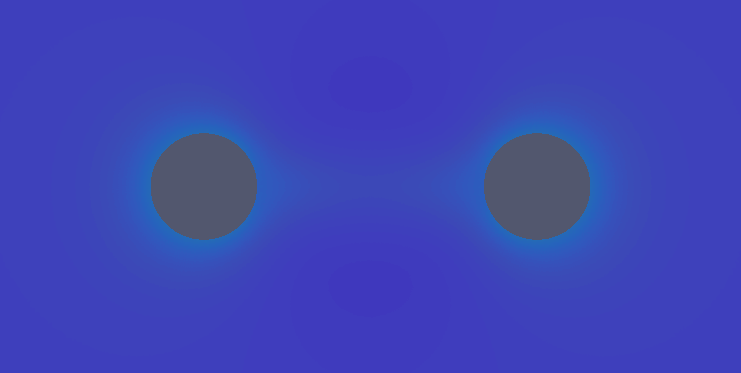} 
  \label{fig:3347_dens} }
   \subfloat[]{\includegraphics[width=4.2cm]{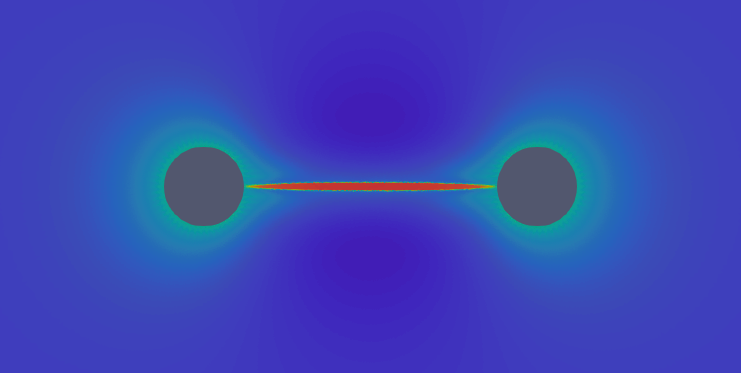}
\label{fig:3348_dens} }
  \subfloat[]{\includegraphics[width=4.2cm]{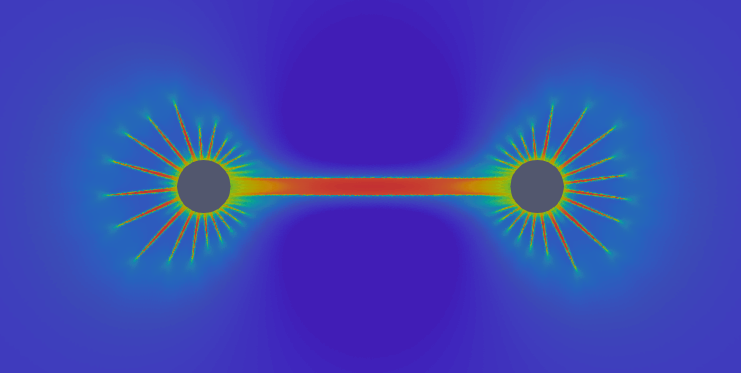}
 \label{fig:3339_dens}}
 \includegraphics[height=2.2cm]{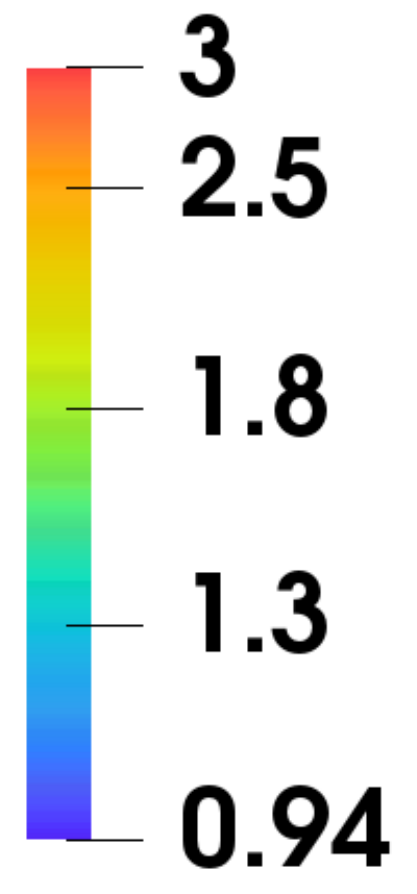}
 \vspace{-0.3cm}
  \caption{Color map of ECM density in the deformed state outside a pair of active particles contracted by
 \protect\subref{fig:3347_dens} $20\%$, \protect\subref{fig:3348_dens} $40\%$,
 \protect\subref{fig:3339_dens} $60\%$,  with $\varepsilon= 5 \cdot10^{-3}r_c$. 
 Particle centers   are located 
 at $(-2.5r_c, 0), (2.5r_c, 0)$. The \tblue{ECM} occupies the disk with center at $(0,0)$ and 
 radius $11r_c$ outside the particles.}
  \label{fig:2cells}
\end{figure}

 \begin{figure}[H]
\centering
 \subfloat[$\varepsilon =0$.]{\includegraphics[width=0.25\linewidth]{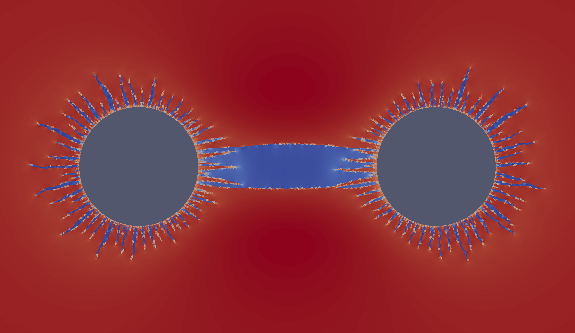} 
 }
\subfloat[$\varepsilon = 5 \cdot10^{-3}r_c $.]{ \includegraphics[width=0.25\linewidth]{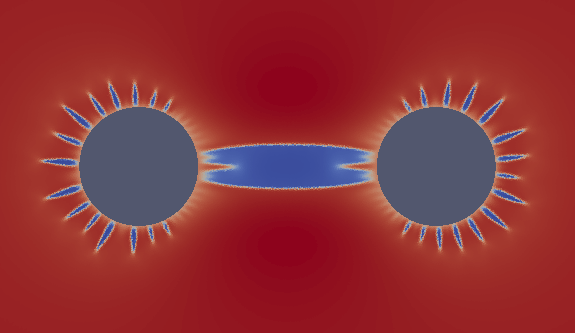}
 }
 \subfloat[$\varepsilon =5 \cdot10^{-2}r_c$]{ \includegraphics[width=0.25\linewidth]{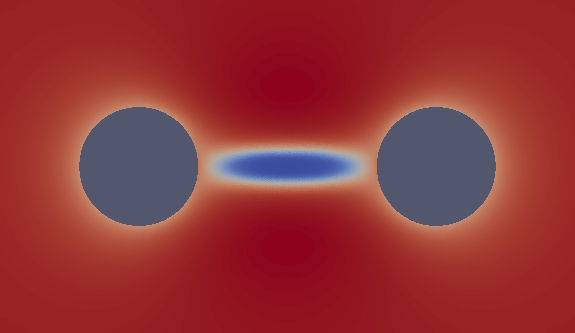}
  }
 \includegraphics[height=2.2cm]{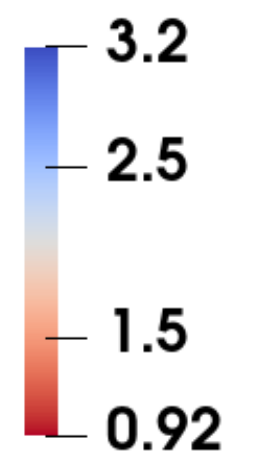}
 \vspace{-0.3cm}
 \caption{ The regularization parameter $\varepsilon$ imposes a length scale: Microstructures at a finer scale are smoothed out.
 For three  $\varepsilon$ values, increasing from left to right, computed density ratio in the
  reference configuration is shown. Active particles with distance between centers
  $5r_c$ contract by $50\%$. 
  }
 \centering
 \label{fig:varying_e}
\end{figure}

\vspace{-0.5cm}
\begin{figure}[h!]
\centering
 \includegraphics[width=0.5\linewidth]{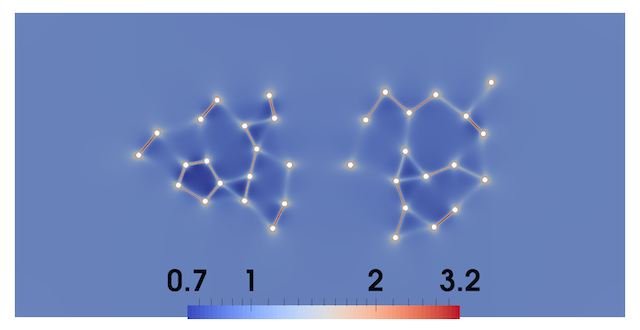}
  \vspace{-0.3cm}
 \caption{ ECM density in the deformed state under multiple active particle contracting.
 Each active particle contracts $50\%$, circles are contained in a rectangular domain, 
 $\varepsilon = 0.045 r_c$. This is similar to the experimental data of \cite[Figure~1D]{shi2014rapid}.
  }
 \centering
 \label{fig:many_cells}
\end{figure}

\vspace{-1cm}
\begin{figure}[H]
\centering
 \subfloat[]{\includegraphics[width=0.33\linewidth]{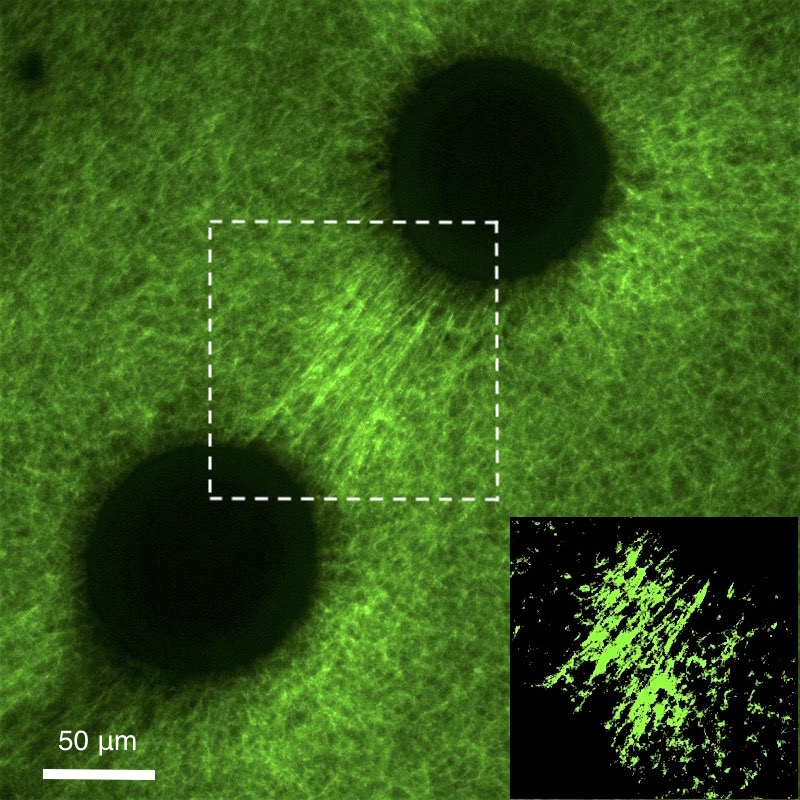} 
 \label{fig:fn3a} 
 }
\subfloat[]{ \includegraphics[width=0.33\linewidth]{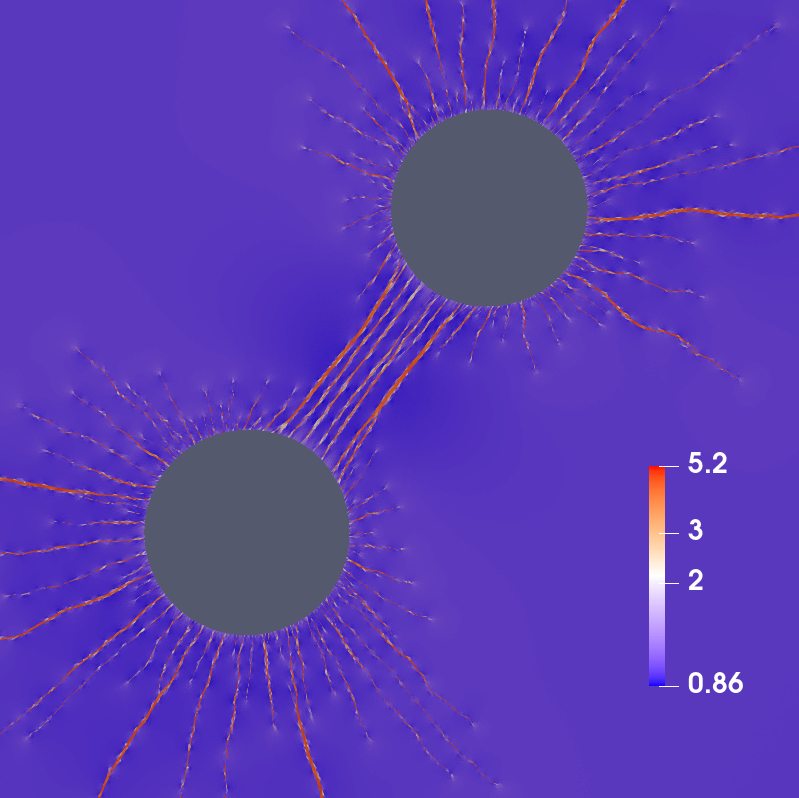}
 \label{fig:fn3b} }
 \\
 \vspace{-0.35cm}
  \subfloat[]{ \includegraphics[width=0.35\linewidth]{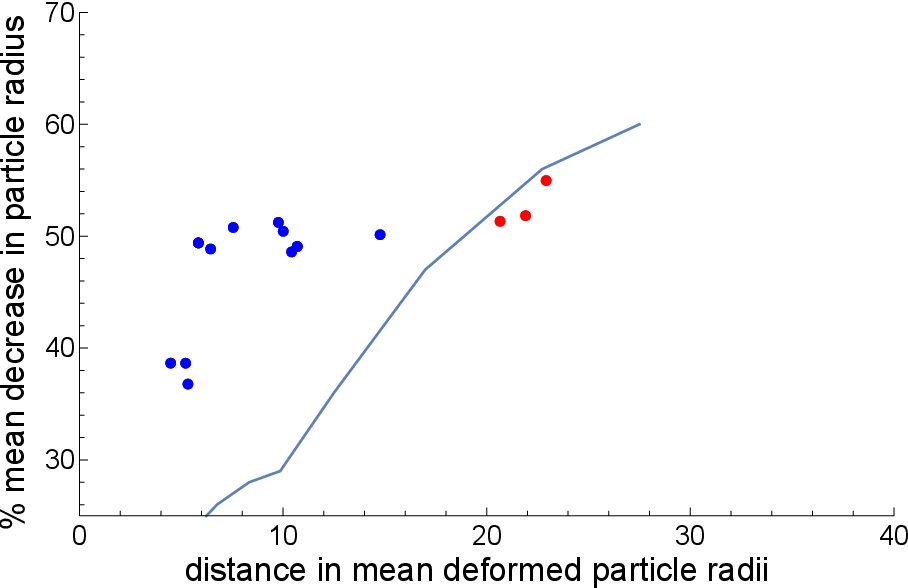}
  \label{fig:fn3c} }
 \subfloat[]{ \includegraphics[width=0.35\linewidth]{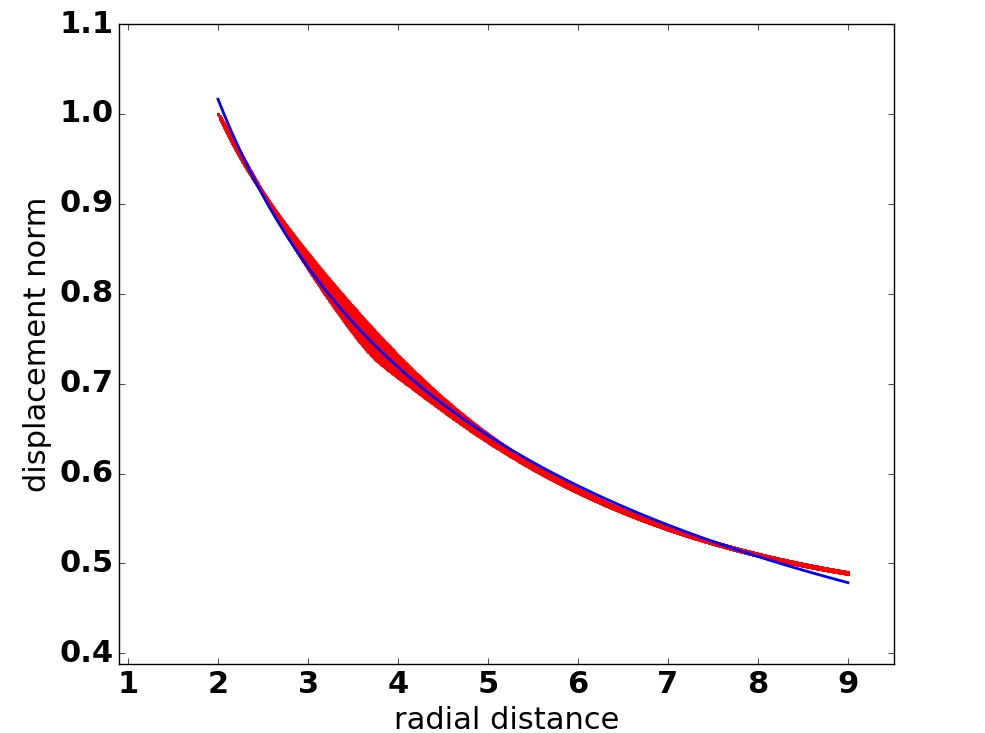}
 \label{fig:fn3d} 
  }
  
 \vspace{-0.3cm}
 \caption{ Comparison with experimental data. (a)-(c) reproduced  from \cite{Grekas_2021} with permission of the authors.
 \protect\subref{fig:fn3a} experiment and \protect\subref{fig:fn3b} simulation of  active particle pair
 contracted by 38\%. \protect\subref{fig:fn3a} A tether fully split into multiple thin bands. 
Insert shows high-contrast version of area within dotted rectangle. \protect\subref{fig:fn3b}
Simulation with initial radii, distances and contractions matched with \protect\subref{fig:fn3a} also
results in split tether. Colorbar: ratio of
deformed to undeformed density.
\protect\subref{fig:fn3c} Predicting whether a tether
forms between two particles. Blue curve: separatrix obtained from multiple simulations. Axes: \% decrease in particle
radius vs deformed distance (in deformed particle radii). Above blue curve, tethers are predicted to form between
particle pairs. No tether is predicted to form below blue curve. Our experimental data (each particle pair is one
point) agreed with the prediction: blue points: tether has formed. Red points: no tether has formed. 
  \protect\subref{fig:fn3d} Displacement norm $|\mbu(\mbx)$ over radial distance ($r=|\mbx|$). Red dots: computed displacement norm at mesh points, blue curve: graph of  $A r^{-n}$ fitted to the red dots. Least squares fit yielded parameters $A = 1.43$ and $n= 0.5$.}
\centering
 \label{fig:exp_comparison}
\end{figure}



\newpage

\section*{Acknowledgment}
Partially supported by the European Union's Horizon 2020 research and innovation programme under the 
Marie Sk\l{}odowska-Curie grant agreement no: 642768 ModCompShock. 
The research of GG was also supported by a Vannevar Bush Postdoctoral Fellowship.

%
%
%
%
%
%
%
 


\bibliographystyle{abbrv}
\bibliography{references.bib,software.bib}

\end{document}